\documentclass[12pt]{amsart}
\usepackage[utf8]{inputenc}	
\usepackage{amssymb,amscd}
\usepackage{multirow,booktabs}
\usepackage[table]{xcolor}
\usepackage{fullpage}
\usepackage{lastpage}
\usepackage{fancyhdr}
\usepackage{mathrsfs}
\usepackage{wrapfig}
\usepackage{graphicx}
\usepackage{setspace}
\usepackage{calc}
\usepackage{enumitem}
\usepackage{multicol}
\usepackage[colorlinks,
            linkcolor=blue,
            anchorcolor=blue,
            citecolor=red,
            ]{hyperref}
\usepackage{cancel}
\usepackage[all]{xy}
\usepackage[margin=3cm]{geometry}

\usepackage{empheq}
\usepackage{framed}
\usepackage{dsfont}
\usepackage{xcolor}
\theoremstyle{plain}
\newtheorem{theorem}{Theorem}[section]
\newtheorem{lemma}[theorem]{Lemma}
\newtheorem{proposition}[theorem]{Proposition}
\newtheorem{corollary}[theorem]{Corollary}

\theoremstyle{remark}
\newtheorem{definition}{Definition}

\newtheorem{remark}{Remark}[section]

\usepackage{indentfirst}
\begin{document}
\setcounter{section}{0}

\thispagestyle{empty}

\newcommand{\QQ}{\mathbb{Q}}
\newcommand{\RR}{\mathbb{R}}
\newcommand{\ZZ}{\mathbb{Z}}
\newcommand{\NN}{\mathbb{N}}
\newcommand{\Nor}{\mathscr{N}}
\newcommand{\CC}{\mathbb{C}}
\newcommand{\HH}{\mathbb{H}}
\newcommand{\EE}{\mathbb{E}}
\newcommand{\Var}{\operatorname{Var}}
\newcommand{\PP}{\mathbb{P}}
\newcommand{\Rd}{\mathbb{R}^d}
\newcommand{\Rn}{\mathbb{R}^n}
\newcommand{\XX}{\mathcal{X}}
\newcommand{\YY}{\mathcal{Y}}
\newcommand{\MM}{\FF}
\newcommand{\BHH}{\overline{\mathbb{H}}}
\newcommand{\XB}{( \mathcal{X},\mathscr{B} )}
\newcommand{\BB}{\mathscr{B}}
\newcommand{\system}{(\Omega,\mathcal{F},\mu,T)}
\newcommand{\FF}{\mathcal{F}}
\newcommand{\MBS}{(\Omega,\mathcal{F})}
\newcommand{\MBSE}{(E,\mathscr{E})}
\newcommand{\MS}{(\Omega,\mathcal{F},\mu)}
\newcommand{\PS}{(\Omega,\mathcal{F},\mathbb{P})}
\newcommand{\LDP}{LDP(\mu_n, r_n, I)}
\newcommand{\Def}{\overset{\text{def}}{=}}
\newcommand{\Series}[2]{#1_1,\cdots,#1_#2}
\newcommand{\independent}{\perp\mkern-9.5mu\perp}
\def\avint{\mathop{\,\rlap{-}\!\!\int\!\!\llap{-}}\nolimits}

\author[Shuo Qin]{Shuo Qin}
\address[Shuo Qin]{Courant Institute of Mathematical Sciences \& NYU-ECNU Institute of Mathematical Sciences at NYU Shanghai}
\email{sq646@nyu.edu}

\author[Pierre Tarrès]{Pierre Tarrès}
\address[Pierre Tarrès]{Courant Institute of Mathematical Sciences \& NYU-ECNU Institute of Mathematical Sciences at NYU Shanghai}
\email{tarres@nyu.edu}

\title{CONTINOUS-TIME VERTEX-REINFORCED RANDOM WALKS ON COMPLETE-LIKE GRAPHS}
\date{}

  \begin{abstract}
    We introduce the continuous-time vertex-reinforced random walk (cVRRW) as a continuous-time version of the vertex-reinforced random walk (VRRW), which might open a new perspective on the
    study of the VRRW.
    
    It has been proved by Limic and Volkov \cite{MR2759737} that for the VRRW on a complete-like graph $K_d \cup \partial K_d$, the asymptotic frequency of visits is uniform over the non-leaf vertices. We give short proofs of those results by establishing a stochastic approximation result for the cVRRW on complete-like graphs. We also prove that almost surely, the number of visits to each leaf up to time $n$ divided by $n^{\frac{1}{d-1}}$ converges to a non-zero limit. We solve a conjecture by Limic and Volkov \cite{MR2759737} on the rate of convergence in the case of the complete graph.
  \end{abstract}
\maketitle

\section{General introduction}

\subsection{Definitions and main results}
Let $G=(V, E)$ be a locally finite connected graph. Each vertex $i \in V$ is equipped with a positive weight $W_i$. We write $i \sim j$ if $(i, j) \in E$. We will sometimes abuse notation, denote the vertex set $V$ by $G$, and identify arbitrary subset $R$ of $G$ to the corresponding subgraph $(R, \sim)$. The outer boundary of $R$ is denoted by
$\partial R:=\{j \in G \backslash R: j \sim R\}$.

Let $X=\left(X_{t}\right)_{t \geq 0}$ be a continuous-time process defined on a probability space $\PS$ and taking values in $V$. For $i \in V$, let $T_i(t)$ be the time $X$ spends at $i$ up to time $t$, that is,
$$
T_{i}(t):=\int_{0}^{t} \mathds{1}_{\{X_{u}=i\}} d u
$$ 
It is often called the local time process of $X$ at $i$. Denote by $\mathbf{T}(t)=(T_i(t))_{i \in V}$ the vector of local times of $X$ at time $t$. We write $T_i(\infty):=\lim_{t \to \infty}T_i(t)$.

For any $\mathbf{T}=\left(T_{i}\right)_{i \in V} \in \RR^V$ and $j \in V$, let $N_{j}(\mathbf{T}):=\sum_{k:k \sim j} T_{k}$. When $G$ is finite, we define
\begin{equation}
  \label{statidisHt}
 \pi_i(\mathbf{T}):=\frac{W_ie^{N_i(\mathbf{T})}}{\sum_{x \in V} W_xe^{N_x(\mathbf{T})}} , \quad i\in V
\end{equation}

\begin{definition}[cVRRW]
  \label{defcvrrw}
  $\left(X_{t}\right)_{t \geqslant 0}$ is called a continuous-time vertex-reinforced random walk (cVRRW) on $G$ with weights $\left(W_{i}\right)_{i \in V}$ and starting point $v_{0} \in V$ if $X_{0}=v_{0}$ and if, at time $t$, $X_t=i\in V$, conditional on the past, it jumps from $i$ to a neighbor $j$ of $i$ at rate $W_{j} e^{N_{j}(\mathbf{T}(t))}$. 
\end{definition}

We often omit “with weights $\left(W_{i}\right)_{i \in V}$" to simplify expressions. The definition of the cVRRW is motivated by its connection to the vertex-reinforced random walk introduced by Diaconis and Pemantle \cite{pemantle1988thesis}.
 
 \begin{definition}[VRRW]
   Given a positive sequence $(a_i)_{i\in V}$, a discrete-time random walk $\left(X_{n}\right)_{n \in \mathbb{N}}$ on $G$ is called a vertex-reinforced random walk (VRRW) with starting point $v_{0} \in V$ and initial local times $(a_i)_{i\in V}$ if $X_{0}=v_{0}$ and, for all $n \in \mathbb{N}$, 
   $$
 \mathbb{P}\left(X_{n+1}=j \mid \mathcal{F}_{n}\right)=\mathds{1}_{\left\{j \sim X_{n}\right\}} \frac{ Z_j(n)}{\sum_{k \sim X_{n}}  Z_k(n)}
 $$
 where $\mathcal{F}_{n}=\sigma\left(X_{0}, \ldots, X_{n}\right)$ and $ Z_i(n):=a_i+\sum_{m=0}^{n} \mathds{1}_{\left\{X_{m}=i\right\}}$.
 \end{definition}

 The observation that the VRRW is a mixture of cVRRWs was first made by Tarr{\`e}s in Lemma 4.7 \cite{tarres2011localization}.
 
 \begin{proposition}[\cite{tarres2011localization}]
   \label{vrrwmixcvrrw}
  A VRRW with initial local times $(a_i)_{i\in V}$ is equal in law to the discrete-time process associated with a cVRRW with independent random weights $W_{i} \sim \operatorname{Gamma}\left(a_{i}, 1\right) .$
 \end{proposition}

 Volkov \cite{volkov2001vertex} conjectured that for a broad class of triangle-free graphs with bounded degree, the VRRW a.s. localizes on a complete $d$-partite graph with possible leaves (see Definition \ref{defEs}). This is our major motivation to investigate the asymptotic behavior of the cVRRW and the VRRW on complete-like graphs or complete $d$-partite graphs with possible leaves: Once that conjecture is proved, the results in this paper could be generalized to a large class of graphs.

A graph $G$ is called a complete-like graph if $G=S\cup \partial S$ where $S=K_d$ $(d\geq 3)$, i.e. complete graph on $d$ vertices, and any $j\in \partial S$ has exactly one neighbor.  We denote its neighbor by $n(j) \in S$. $j$ is then called a leaf of $n(j)$. Note that $\partial S = \emptyset$ is possible, i.e. $G=K_d$. For the cVRRW on a complete-like graph, the main result is the following.

\begin{theorem}
  \label{cvrrwcomlike}
  Let $X$ be a cVRRW on a complete-like graph $G=S \cup \partial S$ with $S=K_d$. Then almost surely, 
  \begin{enumerate}[topsep=0pt, partopsep=0pt, label=(\roman*)]
   \item  $T_j(\infty)<\infty$ and $\lim_{t\to \infty}\pi_j(\mathbf{T}(t))=0$ for all $j\in \partial S$
\item for any $i\in S$, $ \lim_{t\to \infty}\pi_i(\mathbf{T}(t))=\frac{1}{d}$ and $T_i(t)-t/d$ converges.
  \end{enumerate}
\end{theorem}

By Proposition \ref{vrrwmixcvrrw}, we can apply Theorem \ref{cvrrwcomlike} to the VRRW.  For a VRRW on a complete-like graph, let $z(n)=(z_i(n))_{i\in V}$ with $z_i(n):=Z_i(n)/n$.

\begin{corollary}
  \label{VRRWcomlikeas}
  Let $(X_n)_{n\in \NN}$ be a VRRW on a complete-like graph $G=S \cup \partial S$ with $S=K_d$. Then, a.s.
   \begin{enumerate}[topsep=0pt, partopsep=0pt, label=(\roman*)]
\item (Theorem 1, \cite{MR2759737}) for any $i\in S$, $\lim_{n\to \infty}z_i(n)=1/d$,
 \item for any $j \in \partial S$, $\lim_{n\to \infty}n^{-\frac{1}{d-1}}Z_j(n)$ exists and is non-zero.
   \end{enumerate} 
\end{corollary}
\begin{remark}
  The result (ii) for the leaf vertices is new and improves Lemma 9 and Equation (1.4) in Theorem 2 in \cite{MR2759737}.
\end{remark}

These results can be generalized to complete $d$-partite graphs with possible leaves.

\begin{definition}
  \label{defEs}
  Given $d \geq 2$, $(G, \sim)$ will be called a complete $d$-partite graph, if $G=V_{1} \cup \cdots \cup V_{d}$ with 
  \begin{enumerate}[topsep=0pt, partopsep=0pt, label=(\roman*)]
\item $\forall p \in\{1, \ldots, d\}, \forall i, j \in V_{p}$, if $i \neq j$ then $i \nsim j$
\item $\forall p, q \in\{1, \ldots, d\}, p \neq q, \forall i \in V_{p}, \forall j \in V_{q}, i \sim j$
  \end{enumerate} 
$G$ is called a complete d-partite graph with (possible) leaves, if $G=S\cup \partial S$ where $S=V_{1} \cup \cdots \cup V_{d}$ is a d-partite graph and any $j\in \partial S$ has exactly one neighbor. 

\end{definition}

\begin{theorem}
  \label{comdpartthm}
  Let $G=S \cup \partial S$ be a complete d-partite $(d\geq 3)$ graph with possible leaves and $X$ be a cVRRW on $G$. Then, almost surely, 
  \begin{enumerate}[topsep=0pt, partopsep=0pt, label=(\roman*)]
\item $T_j(\infty)<\infty$ and $\lim_{t\to \infty}\pi_j(\mathbf{T}(t))=0$ for all $j\in \partial S$. 
\item for any $i \in S$, $z_i:= \lim_{t\to \infty}\pi_i(\mathbf{T}(t))$ exists and is positive. In addition, for any $p\in \{1,2,\cdots,d\}$, $\sum_{i\in V_p}z_i=1/d$. 
  \end{enumerate}  
Moreover, let $(X_n)_{n\in \NN}$ be a VRRW on $G$. Then, a.s., 
\begin{enumerate}[topsep=0pt, partopsep=0pt, label=(\Roman*)]
\item (Section 4, \cite{MR2759737}) for any $i\in S$, $\tilde{z}_i:=\lim_{n\to \infty}Z_i(n)/n$ exists and is positive. For any $p\in \{1,2,\cdots,d\}$, $\sum_{i\in V_p}\tilde{z}_i=1/d$.
\item for any $j \in \partial S$, $$\lim_{n\to \infty} \frac{\log Z_j(n)}{\log n}  = \frac{d }{d-1}\tilde{z}_{n_j}$$
\end{enumerate}   
\end{theorem}

For a complete-like graph $G=S \cup \partial S$, we define $z^*=(z_i^*)_{i\in V}$ by 
\begin{equation}
  \label{zstardef}
  z^*_i:=\frac{1}{d}, \quad \text{if}\ i\in S; \quad z^*_i:=0 \quad \text{if}\ i\in \partial S
\end{equation}
Theorem \ref{cvrrwcomlike} implies the almost sure convergence of $(\pi_i(\mathbf{T}(t)))_{i \in V}$ to $z^*$.
The following theorem gives upper bounds on the rate of convergence, which improves Theorem 2 in \cite{MR2759737}. 

For $x\in \RR^V$, we write $\|x\|:=\|x\|_{L^2}$ for simplicity.

\begin{theorem}
  \label{rateupper}
  Let $(X_t)_{t\geq 0}$ and $(X_n)_{n\in \NN}$, respectively, be a cVRRW and a VRRW on $G$ where $G=S\cup \partial S$ is a complete-like graph with $S=K_d$. 
  \begin{enumerate}[topsep=0pt, partopsep=0pt, label=(\roman*)]
 \item If $d>3$, then almost surely,
  \begin{equation}
    \label{convratedlarge3}
  \lim_{t\to\infty} \|\pi(\mathbf{T}(t))-z^*\|e^{\frac{t}{d}} \quad \text{and}\quad \lim_{n\to\infty} \|z(n)-z^*\|n^{\frac{1}{d-1}} \quad \text{exist}
  \end{equation}
and for any $i\in S$,
  \begin{equation}
    \label{limTiminustd}
     \lim_{t\to \infty}\left( T_i(t)-\frac{t}{d}\right) =\log W_i+ \sum_{v\in \partial S: v\sim i}T_{v}(\infty) -\frac{1}{d}(\sum_{x\in S}\log W_x+ 2\sum_{j\in \partial S}T_j(\infty))
  \end{equation}
 \item  If $d=3$, then for any $\varepsilon>0$, almost surely,
\begin{equation}
  \label{convrated3}
\lim_{t\to\infty} \|\pi(\mathbf{T}(t))-z^*\|e^{\frac{(1-\varepsilon)t}{3}}=0  \quad \text{and}\quad \lim_{n\to\infty} \|z-z^*\|n^{\frac{1}{2}-\varepsilon}=0 
\end{equation}
  \end{enumerate}
\end{theorem}
\begin{remark}
 (a)  The exponentially fast convergence can be generalized to complete d-partite graphs with possible leaves by similar arguments. In particular, Theorem \ref{comdpartthm} (II) can be improved: Almost surely, $ Z_j(n)/n^{\frac{d }{d-1}\tilde{z}_{n_j}}$ converges. \\
 (b) If $G=K_d$, i.e. $\partial S=\emptyset$, the right hand side of (\ref{limTiminustd}) is deterministic.
  \end{remark}

 In the case $G=K_d$ with $d>3$, we prove a conjecture by Limic and Volkov \cite{MR2759737} that the exponent $1/(d-1)$ in (\ref{convratedlarge3}) is optimal in the following sense.

 \begin{theorem}
  \label{d-1optimal}
  Let $(X_t)_{t\geq 0}$ and $(X_n)_{n\in \NN}$, respectively, be a cVRRW and a VRRW on $G=K_d$ with $d>3$, and let $\varepsilon>0$ be a constant, then
  \begin{equation}
   \label{opticVRRW}
        \PP(\lim_{t\to\infty}\|\pi(\mathbf{T}(t))-z^*\|e^{\frac{(1+\varepsilon)t}{d}}=\infty)>0,\quad \PP(\lim_{n\to\infty}\|z(n)-z^*\|n^{\frac{1}{d-1}+\varepsilon}=\infty)>0
  \end{equation}
 \end{theorem}
  
If $G$ is a triangle, we improve the upper bound in (\ref{convrated3}), which disproves a conjecture by Limic and Volkov \cite{MR2759737} on the true rate of convergence.

\begin{proposition}
  \label{K3lb}
  Let $(X_t)_{t\geq 0}$ and $(X_n)_{n\in \NN}$, respectively, be a cVRRW and a VRRW on $G=K_3$ and $\kappa>1/2$ be a constant, then
  \begin{equation}
    \label{limsup3cVRRW}
    \PP(\lim_{t\to\infty}\frac{\|\pi(\mathbf{T}(t))-z^*\|e^{\frac{t}{3}}}{t^{\kappa}}=0)=1,\quad  \PP(\limsup_{t\to\infty}\frac{\|\pi(\mathbf{T}(t))-z^*\|e^{\frac{t}{3}}}{\sqrt{t}}=\infty)=1
  \end{equation}
and
 \begin{equation}
  \label{limsup3VRRW}
  \PP(\lim_{n\to\infty}\frac{\|z(n)-z^*\|\sqrt{n}}{(\log n)^{\kappa}}=0)=1,\quad  \PP(\limsup_{n\to\infty}\frac{\|z(n)-z^*\|\sqrt{n}}{\sqrt{\log n}}=\infty)=1
\end{equation}
\end{proposition}

\subsection{Organization of the paper} In Section \ref{timelinesconsec}, we describe a timelines construction of the cVRRW which can be used to prove Proposition \ref{vrrwmixcvrrw}. Using the Markov property of $\left(X_{t}, \mathbf{T}(t)\right)_{t\geq 0}$, we establish in Section \ref{conMJPsec} a stochastic equation for $f(\mathbf{T}(t))$, see Proposition \ref{propfTstoapp}. In Section \ref{cVRRWcomsec}, for the case $G$ is a complete graph, we give a short proof of Theorem \ref{cvrrwcomlike} and retrieve a result by Pemantle \cite{pemantle1988thesis}. Section \ref{comlikesub} is devoted to the proof of Theorem \ref{cvrrwcomlike}, Corollary \ref{VRRWcomlikeas} and Theorem \ref{comdpartthm}. We also obtain a stochastic approximation result for the cVRRW on complete-like graphs, i.e. Theorem \ref{fTasconv}, which is of independent interest. The proofs of Theorem \ref{rateupper}, Theorem \ref{d-1optimal} and Proposition \ref{K3lb} are presented in Section \ref{rateconsub}.

\section{Timelines construction}
\label{timelinesconsec}

We may construct the cVRRW from a collection of random alarms similar to Rubin's algorithm for ordinal-dependent urns, see e.g. \cite{MR1030727}.
\begin{definition}[Timelines construction]
  \label{timelinesdef}
  Let $(X_t)_{t\geq 0}$ be defined as follows: 
  \begin{enumerate}[topsep=0pt, partopsep=0pt, label=(\roman*)]
 \item Define, on each vertex $i \in V$, independent point processes (alarm times) 
    $$
    0<\Theta_{1}^{(i)} < \Theta_{2}^{(i)} < \Theta_{3}^{(i)}\cdots
    $$
    with the convention that $\Theta_{1}^{(i_0)}=0$ if $X_0=i_0$.
    \item Each vertex $i \in V$ has its own clock, denoted by $\tilde{T}_{i}(t)$, which only runs when the process $\left(X_{t}\right)_{t \geq 0}$ is adjacent to $i$. That is, 
$\tilde{T}_{i}(t)=\sum_{k: k\sim i} T_{k}(t)$.
\item When the clock of a vertex $i \in V$ rings, i.e. when $\tilde{T}_{i}(t)=\Theta_{k}^{(i)}$ for some $k>0$, then $X_{t}$ jump to crosses $i$ instantaneously (this can happen only when $X\sim i$).
  \end{enumerate} 
Then $X$ is called the $(\Theta_{k}^{(i)})_{i\in V,k\in \NN^+}$-induced continuous-time random walk.
\end{definition}

\begin{remark}
 $\Theta^{(i)}_n$ is the time that $X$ needs to spend on the neighboring vertices of $i$ before the $n$-th jump to $i$. Given $(\Theta_{k}^{(i)})_{i\in V,k\in \NN^+}$, $X$ is indeed deterministic.
\end{remark}

 For $T\geq 0$ and $i \in V$, we define a distribution function by
$$
F^{(T)}_{W_i}(t)= \begin{cases}1-\frac{1}{e^{W_i(e^t-e^T)}}, &  t \geq T \\ 0, & t<T\end{cases}
$$
A comparison to the exponential distribution leads to the following.
\begin{proposition}
  \label{propeFtwi}
  Suppose that $\xi \sim F^{(T)}_{W_i}$, then $\xi > T$ a.s., and for $t>T$, 
\begin{equation}
  \label{newdistrirate}
  \lim_{\Delta t \to 0}\frac{\PP(t<\xi \leq t+\Delta t|\xi>t)}{\Delta t}=\frac{1}{1-F^{(T)}_{W_i}(t)}\frac{d F^{(T)}_{W_i}(t)}{dt}=W_ie^t
\end{equation}
The distribution function of $\xi-T$ is 
$$
F_{W_i,T}(a)= \begin{cases}1-\frac{1}{e^{W_ie^T(e^a-1)}}, &  a \geq 0 \\ 0, & a<0\end{cases}
$$
and in particular, $\EE (\xi - T)\leq 1/(W_ie^T)$ and $\EE (\xi - T)^2\leq 2/(W_ie^T)^2$.
\end{proposition}

\begin{proposition}
  \label{proptimeline1}
On each site $i$, sample $\Theta_{1}^{(i)} \sim F^{(0)}_{W_i}$. For each $n\in \NN^+$, given $\Theta_{n}^{(i)}$, sample $\Theta_{n+1}^{(i)} \sim F^{(\Theta_{n}^{(i)})}_{W_i}$. Let $X$ be the $(\Theta_{k}^{(i)})_{i\in V,k\in \NN^+}$-induced continuous-time random walk as in Definition \ref{timelinesdef}. Then $X$ is a cVRRW with weights $\left(W_{i}\right)_{i \in V}$.
\end{proposition}
\begin{proof}
  This is a direct consequence of (\ref{newdistrirate}) and Definition \ref{defcvrrw}.
\end{proof}
Note that if $\{U(t), t \geq 0\}$ is a Poisson process with unit rate, then $\{U\left(W_x\left(e^t-1\right)\right), t\geq 0\}$ is a point process with intensity measurable $W_x e^t d t$. This gives another way to generate $(\Theta_n^{(i)})$ in Proposition \ref{proptimeline1}.
\begin{corollary}
  \label{wevexpsum}
  Let $\{U_x(t), t\geq 0\}_{x\in V}$ be independent unit-rate Poisson processes. Define a continuous time random walk $(\widetilde{X})_{t\geq 0}$ as follows: it jumps to $x \in V$ at time $t$ if $$U_x(W_x(e^{N_x(\widetilde{\mathbf{T}}(t))}-1))=U_x(W_x(e^{N_x(\widetilde{\mathbf{T}}(t-))}-1))+1$$ where $\widetilde{\mathbf{T}}(t)$ is its local times at time $t$. Then $(\widetilde{X})_{t\geq 0}$ is a cVRRW with weights $(W_x)_{x\in V}$. In particular, let $\xi_1,\xi_2,\cdots$ be i.i.d. Exp(1)-distributed random variables, define $\tilde{\Theta}_{n}^{(i)}$ by
  $$
  W_i(e^{\tilde{\Theta}_{n}^{(i)}}-1)=\xi_1+\xi_2+\cdots+\xi_n, \quad n=1,2,\cdots
  $$
  Let $\{\Theta_{n}^{(i)}\}$ be as in Proposition \ref{proptimeline1}. Then, $\{\Theta_n^{(i)}\}$ and $\{\tilde{\Theta}_{n}^{(i)}\}$ have the same law.
\end{corollary}
\begin{proof}
  The first assertion follows from the definition of the cVRRW. Note that the second assertion could also be deduced from the fact that given $(\Theta_{1}^{(i)}, \Theta_{2}^{(i)}, \cdots, \Theta_{n}^{(i)}) $, $W_i(e^{\Theta_{n+1}^{(i)}}-e^{\Theta_{n}^{(i)}})$ has an exponential distribution with unit rate. 
\end{proof}

 We may embed a VRRW into a continuous time process. We construct a continuous-time process $\widetilde{X}_t$ as in Definition \ref{timelinesdef} where for $i\in V$, we set 
\begin{equation}
  \label{contvrrwvkexp}
  \Theta_{k}^{(i)}=\sum_{l=0}^{k-1} \frac{1}{a_{i}+l} \xi_{l}^{(i)}, \quad \forall k \geq 1
\end{equation}
where $\left(\xi_{k}^{(i)}\right)_{i \in V, k \geq 1}$ are i.i.d. Exp(1)-distributed random variables. Let $\eta_{n}$ be the $n$-th jump time of $\left(\widetilde{X}_{t}\right)_{t \geq 0}$, with the convention that $\eta_{0}:=0$. Tarr{\`e}s \cite{tarres2011localization} proved the following lemma which says that $\left(\tilde{X}_{\eta_{n}}\right)_{n \geq 0}$ is a VRRW with the initial local times $(a_i)_{i\in V}$. 
\begin{lemma}
  \label{lemtimelinevrrw}
  Let $\left(X_{n}\right)_{n \in \mathbb{N}}$ be a VRRW with initial local times $(a_i)_{i\in V}$ and let $\left(\widetilde{X}_{t}\right)_{t \geq 0}$ be as above. Assume that they start at some vertex $i_{0} \in V$. Then $\left(\widetilde{X}_{\eta_{n}}\right)_{n \geq 0}$ and $\left(X_{n}\right)_{n \geq 0}$ have the same distribution.
\end{lemma}
\begin{proof}
  The proof is similar to that of Lemma 3.2 in \cite{tarres2011localization}.
\end{proof}

 \begin{proof}[Proof of Proposition \ref{vrrwmixcvrrw}]
  Given Lemma \ref{lemtimelinevrrw}, the proof is then similar to that of Lemma 4.7 in \cite{tarres2011localization} and Theorem 1 in \cite{MR3420510}.
\end{proof}

\section{Connection with Markov jump process}
\label{conMJPsec}

It will be helpful to view the cVRRW $(X_{t})_{t\geq 0}$ as the marginal of the process $\left(X_{t}, \mathbf{T}(t)\right)_{t\geq 0}$, since the joint process $\left(X_{t}, \mathbf{T}(t)\right)_{t\geq 0}$ is a Markov process. Indeed, it is easy to check that $\left(X_{t}, \mathbf{T}(t)\right)_{t\geq 0}$ is a Feller process. We complete the natural filtration $\sigma(X_s:0\leq s\leq t)$ by adding all the $\PP$-negligible sets. We denote the completed filtration by $(\FF_t)_{t\geq 0}$ with respect to which the Markov property is preserved, see e.g. Proposition 2.14, Chapter 3 in \cite{MR1083357}.  Moreover, by Proposition 2.10 in Chapter 3 in \cite{MR1083357}, $(\FF_t)_{t\geq 0}$ is right-continuous.

When $G$ is finite, for fixed $\mathbf{T}\in \RR^{V}$, we consider a continuous-time Markov chain $A=(A_t)_{t\geq 0}$ with jump rate from $i$ to $j$ equal to $W_je^{N_j(\mathbf{T})}$. That is, its generator $L(\mathbf{T})$ is given by 
\begin{equation}
  \label{LTdef}
  L(\mathbf{T})(g)(i) = \sum_{j: j\sim i} W_j e^{N_j(\mathbf{T})}(g(j) - g(i)), \quad g \in \RR^V
\end{equation}

By a slight abuse of notation, we also use the notation $L(\mathbf{T})$ for the $|V| \times |V|$ matrix of that operator w.r.t. the canonical basis $\{e_i:=\left(1_{j=i}\right)_{j \in V}\}_{i\in V}$. That is, $L(\mathbf{T})(i,j)=W_j e^{N_j(\mathbf{T})}$ if $i\sim j$, and equals 0 otherwise. Note that the generator of the cVRRW $(X_t,\mathbf{T}(t))$ is 
\begin{equation}
  \label{twogenerpartialT}
 (\mathcal{L} f)(i, \mathbf{T})=L(\mathbf{T})f(\cdot, \mathbf{T}))(i)+\frac{\partial f(i, \mathbf{T})}{\partial T_{i}} 
\end{equation}

\begin{proposition}
  \label{cmcstadis}
 $A$ is irreducible, positive recurrent, and reversible. The stationary distribution is given by $(\pi_i(\mathbf{T}))_{i\in V}$ defined in (\ref{statidisHt}).
\end{proposition}
\begin{proof}
  Simply observe that for any $g\in \RR^V$, $\int L(\mathbf{T}) g d \pi(\mathbf{T})=0$ by (\ref{LTdef}).
\end{proof}

Let $I$ be the identity matrix and define the $|V| \times |V|$ matrix $\Pi(\mathbf{T})$ by 
$$
\Pi(\mathbf{T})_{i,j}=\pi_j(\mathbf{T}), \quad i,j \in V
$$
As in Section 2.2.3, \cite{aldous-fill-2014}, let us define, for all $\mathbf{T} \in \mathbb{R}^{V}$, 
\begin{equation}
  \label{QTdef}
  Q(\mathbf{T}):=-\int_{0}^{\infty}\left(e^{u L(\mathbf{T})}-\Pi(\mathbf{T})\right) d u
\end{equation}
which exists since the matrix of transition probabilities $P_u:=e^{u L(\mathbf{T})}$ converges towards $\Pi(\mathbf{T})$ at exponential rate when $|V|$ is finite. Similar notions have been used in the study of the VRRW, see Bena{\"\i}m \cite{MR1428513}, and Bena{\"\i}m, Tarr{\`e}s \cite{MR2932667}. Sabot and Tarr{\`e}s \cite{MR3420510} also used $Q(\mathbf{T})$ to study the vertex-reinforced jump process.

By definition, $Q(\mathbf{T})$ is a solution of the Poisson equation for the Markov chain $L(\mathbf{T})$,
\begin{equation}
  \label{poissonsolu}
 L(\mathbf{T}) Q(\mathbf{T})=Q(\mathbf{T}) L(\mathbf{T})=I-\Pi(\mathbf{T})
\end{equation}
Moreover, we have $\Pi(\mathbf{T}) Q(\mathbf{T})= Q(\mathbf{T})\Pi(\mathbf{T})=0$. In the sequel, we denote $E_{i}^{\mathbf{T}}$ the expectation with respect to the law of $(A)_{t\geq 0}$ defined by (\ref{LTdef}) with starting point $i\in V$.

\begin{proposition}
\label{QTformulatau}
  For all $\mathbf{T} \in \mathbb{R}^{V}$ and $i, j \in V$,  let $\tau_{j}$ be the first hitting time of $(A)_{t\geq 0}$ to vertex $j$ and write $\pi=\pi(\mathbf{T})$ for simplicity. Then,
  \begin{equation}
    \label{QTij}
   Q(\mathbf{T})_{i, j}=\pi_j \sum_{r \in V} \pi_r(E_{i}^{\mathbf{T}}\left(\tau_{j}\right)- E_{r}^{\mathbf{T}}\left(\tau_{j}\right)) 
  \end{equation}
\end{proposition}
\begin{proof}
  By the strong Markov property of $A$, one can show that 
  \begin{equation}
    \label{Qijexpress}
   Q(\mathbf{T})_{i, j}=\pi_j E_{i}^{\mathbf{T}}\left(\tau_{j}\right)+Q(\mathbf{T})_{j, j} 
  \end{equation}
  On the other hand,
  \begin{equation}
    \label{Qjj}
    Q(\mathbf{T})_{j, j}=\sum_{i \in V} \pi_iQ(\mathbf{T})_{j, j}=\sum_{i \in V} [\pi_i Q(\mathbf{T})_{i, j} - \pi_i\pi_j E_{i}^{\mathbf{T}}\left(\tau_{j}\right)]= -\pi_j\sum_{i \in V} \pi_i E_{i}^{\mathbf{T}}\left(\tau_{j}\right)
  \end{equation}
  where we used $\Pi(\mathbf{T})Q(\mathbf{T})=0$. Then, (\ref{QTij}) follows from (\ref{Qijexpress}) and (\ref{Qjj}).
\end{proof}
\begin{remark}
Proposition \ref{QTformulatau} should be compared to Lemma 2.11 and Lemma 2.12 in \cite{aldous-fill-2014} which are referred to as the mean hitting time formula for finite Markov chains there. 

Sabot and Tarr{\`e}s in Section 4.1 \cite{MR3420510} proved the case when the stationary distribution is the uniform distribution.
\end{remark}

\begin{proposition}
  \label{propfTstoapp}
  Let $(X_t)_{t\geq 0}$ be a cVRRW on a finite connected graph $G$ and 
  $$f: \RR^V \to \RR, \quad \mathbf{T} \mapsto f(\mathbf{T})$$ 
  be a twice continuously differentiable function function. Then,
  \begin{equation}
    \label{fTstoapp}
    \begin{aligned}
    &\quad f(\mathbf{T}(t))-f(\mathbf{T}(0))-\int_0^t  \pi(\mathbf{T}(u))\cdot \nabla f(\mathbf{T}(u)) d u\\
    &=[Q(\mathbf{T}(t))\nabla f(\mathbf{T}(t))]_{X_{t}}-[Q(\mathbf{T}(0))\nabla f(\mathbf{T}(0))]_{X_{0}}+M_{f}(t) \\
      &-\int_{0}^{t} [\frac{\partial Q(\mathbf{T}(u))}{\partial T_{X_{u}}} \nabla f(\mathbf{T}(u))]_{X_{u}} d u -\int_{0}^{t} [Q(\mathbf{T}(u))\frac{\partial }{\partial T_{X_{u}}}\nabla f(\mathbf{T}(u)) ]_{X_{u}} d u
      \end{aligned}
  \end{equation}
where $M_f(t)$ is a martingale with quadratic variation $\left\langle M_f\right\rangle_t=\int_0^t g_f\left(X_s, \mathbf{T}(s)\right) d s$ where
\begin{equation}
    \label{defgxtgeneral}
    g_{f}(x,\mathbf{T}):=\sum_{p: p\sim x} W_{p}e^{N_p(\mathbf{T})}([Q(\mathbf{T})\nabla f(\mathbf{T})]_{p}-[Q(\mathbf{T})\nabla f(\mathbf{T})]_{x})^2, \quad x\in V, \mathbf{T} \in \RR^V
  \end{equation}
\end{proposition}
\begin{proof}
First note that $f(\mathbf{T}(t))$ is absolutely continuous in $t$ and 
  $$
   f(\mathbf{T}(t))-f(\mathbf{T}(0))=\int_0^t \frac{d}{du}f(\mathbf{T}(u))du=\int_0^t \frac{\partial}{\partial T_{X_u}}f(\mathbf{T}(u))du=\int_0^t \sum_{x\in V} \frac{\partial }{\partial T_{x}}f(\mathbf{T}(u))\mathds{1}_{X_u=x} du
  $$
  Thus, the left-hand side of (\ref{fTstoapp}) equals
  \begin{equation}
    \label{stoappequf}
\int_0^t \sum_{x\in V}(\mathds{1}_{X_u=x}-\pi_x(\mathbf{T}(u))) \frac{\partial }{\partial T_{x}}f(\mathbf{T}(u)) du 
        =\int_{0}^{t} [(I-\Pi(\mathbf{T}(u)))\nabla f(\mathbf{T}(u))]_{X_u}du
  \end{equation}
  which, by (\ref{twogenerpartialT}) and (\ref{poissonsolu}), equals
  \begin{equation}
    \label{stoappequf1}
    \begin{aligned}
     &\quad \int_0^t(L(T(u)) Q(\mathbf{T}(u))\nabla f(\mathbf{T}(u)))_{X_u} d u\\
      &=\int_{0}^{t} \mathcal{L}\left(Q\nabla f(\cdot)_{\cdot }\right)\left(X_{u}, \mathbf{T}(u)\right) d u-\int_{0}^{t} \frac{\partial}{\partial T_{X_{u}}} Q\nabla f(\mathbf{T}(u))_{X_{u}} d u \\
  &=Q\nabla f(\mathbf{T}(t))_{X_{t}}-Q\nabla f(\mathbf{T}(0))_{X_{0}}+M_{f}(t)-\int_{0}^{t} \frac{\partial}{\partial T_{X_{u}}} Q\nabla f(\mathbf{T}(u))_{X_{u}} d u
    \end{aligned}
  \end{equation}
   where we wrote $Q\nabla f(\mathbf{T}(u))=Q(\mathbf{T}(u))\nabla f(\mathbf{T}(u))$ for simplicity of notation and 
  $$
  M_{f}(t):=-Q\nabla f(\mathbf{T}(t))_{X_{t}}+Q\nabla f(\mathbf{T}(0))_{X_{0}}+\int_{0}^{t} \mathcal{L}\left(Q\nabla f(\cdot)_{\cdot}\right)\left(X_{u}, \mathbf{T}(u)\right) d u
  $$
  is a martingale by Dynkin's formula. It remains to prove (\ref{defgxtgeneral}). First note that for any $t>0$, $|\mathcal{L}\left([Q\nabla f(\cdot)_{\cdot}]^2\right)|$ and $|\mathcal{L}\left(Q\nabla f(\cdot)_{\cdot}\right)|$ are bounded on $V \times [0,t]^{V}$ by our assumptions. Since $\left(X_u, \mathbf{T}(u)\right) \in V \times[0, t]^V$ for $u \leq t$, by applying Dynkin's formula to $[Q\nabla f(\mathbf{T}(t))_{X_{t}}]^2$, we have 
\begin{equation}
  \label{M2expbd}
    |\EE(M_f(r)^2-M_f(s)^2|\FF_s )|= |\EE([M_f(r)-M_f(s)]^2|\FF_s )|\leq K(t)(r-s), \quad 0\leq s\leq r \leq t
\end{equation}
for some constant $K(t)$. Ethier–Kurtz criterion (Corollary 7.4, Page 83, \cite{MR0838085}) then enables us to write $\langle M_{f}\rangle_t=\int_0^t\frac{d}{d s}\langle M_{f}\rangle_sds$ with $|\frac{d}{d s}\langle M_{f}\rangle|\leq K(t)$ for $s\leq t$. Note that $\frac{d}{d s}\langle M_{f}\rangle$ is $\FF_{s+}$ measurable. By the right-continuity of $(\FF_t)$ and the dominated convergence theorem, 
$$
\begin{aligned}
\frac{d}{d s}\langle M_{i}\rangle_{s} &= \EE(\lim_{\varepsilon\to 0+}\frac{\langle M_{i}\rangle_{s+\varepsilon}-\langle M_{i}\rangle_{s}}{\varepsilon}|\FF_s )= \lim_{\varepsilon\to 0+}\EE(\frac{\langle M_{i}\rangle_{s+\varepsilon}-\langle M_{i}\rangle_{s}}{\varepsilon}|\FF_s )\\ &= \lim_{\varepsilon\to 0+} \mathbb{E}\left(\frac{\left(M_{i}(s+\varepsilon)-M_{i}(s)\right)^{2}}{\varepsilon} \mid \mathcal{F}_{s}\right) \\
&=\left(\frac{d}{d \varepsilon} \mathbb{E}\left(\left(Q(\mathbf{T}(t+\varepsilon))_{X_{s+\varepsilon},i}-Q(\mathbf{T}(t))_{X_{s},i}\right)^{2} \mid \mathcal{F}_{s}\right)\right)_{\varepsilon=0} \\
&=\mathcal{L}\left([Q]^2(\cdot)_{\cdot,i}\right)\left(X_{s}, \mathbf{T}(s)\right)-2 Q(\mathbf{T}(s))_{X_s,i}\mathcal{L}\left(Q(\cdot)_{\cdot,i}\right)\left(X_{s}, \mathbf{T}(s)\right) \\
&=\sum_{p: p \sim X_s} W_p e^{N_p(\mathbf{T}(s))}\left(Q(\mathbf{T}(s))_{p, i}-Q(\mathbf{T}(s))_{X_s, i}\right)^2
\end{aligned}
$$
which completes the proof of (\ref{defgxtgeneral}).
\end{proof}

\section{Proofs of the main results}

When $G$ is finite, for any probability measure $y=(y_i)_{i\in V}$ on $V$ (i.e. $\sum_{i\in V}y_i=1$ and $y_i\geq 0$ for all $i\in V$), define $H(y):=\sum_{i\in V} y_i N_i(y)$. If $\mathbf{T}(t)$ is the local time process of a cVRRW, for simplicity of notation, we write $H(t):=H(\pi(\mathbf{T}(t)))$.

\subsection{cVRRW on complete graphs} 
\label{cVRRWcomsec} When $G$ is a complete graph, we have explicit formulas for $Q(\mathbf{T})_{i,j}$ defined in (\ref{QTdef}).
\begin{lemma}
  \label{QKd}
  Let $G=K_d$ $(d\geq 3)$ and fix $\mathbf{T} \in \RR^d$. Then, for any $x,y\in V$ such that $x\neq y$
  \begin{equation}
    \label{exphitKd}
    E_x^{\mathbf{T}} \tau_y = \frac{1}{ W_y e^{N_y(\mathbf{T})}}
  \end{equation}
and for $i,j\in V$,
$$
Q(\mathbf{T})_{i,j}=\frac{\pi_j(\mathbf{T})}{\sum_{x\in V}W_x e^{N_x(\mathbf{T})}}, \quad i\neq j; \quad Q(\mathbf{T})_{j,j}=-\frac{1-\pi_j(\mathbf{T})}{\sum_{x\in V}W_x e^{N_x(\mathbf{T})}}
$$
In particular, 
\begin{equation}
  \label{QestKd}
  \sup_{x,i\in V} \left( |Q(\mathbf{T})_{x,i}|, |\frac{\partial}{\partial T_{x}} Q(\mathbf{T})_{x, i}|\right) \leq \frac{1}{\sum_{x\in V}W_x e^{N_x(\mathbf{T})}} 
\end{equation}
\end{lemma}
\begin{proof}
  (\ref{exphitKd}) follows from the definition of the Markov chain $A$. The rest follows from (\ref{QTij}).
\end{proof}

 For $\mathbf{T} \in \RR^V$, define
$$
V(\mathbf{T}):=\frac{1}{d}\sum_{i\in V}\log \pi_i(\mathbf{T})=- \log \sum_{x\in V}W_xe^{N_x(\mathbf{T})} +\frac{1}{d}\sum_{i\in V}\left[N_i(\mathbf{T})+\log W_i\right]
$$
Note that $-\log d -V(\mathbf{T})$ is a relative entropy and thus non-negative. It equals 0 iff $\pi(\mathbf{T})$ is the uniform measure. If $(\mathbf{T}(t))$ is the local time process of a cVRRW, for simplicity of notation, we write
 $$V(t):=V(\mathbf{T}(t))$$

Informally, one might regard $V(t)$ as a Lyapunov function for the dynamical system governing $\pi(\mathbf{T}(t))$. A more general form of $V(t)$ was originally introduced by Losert and Akin in 1983 in \cite{MR0714271} in the study of the deterministic Fisher-Wright-Haldane population genetics model, and used by Bena{\"\i}m and Tarr{\`e}s \cite{MR2932667} in the study of VRRW. In the proof of the following result, we use a technique originally introduced by Métivier and Priouret \cite {MR0873887}, and adapted by Bena{\"\i}m \cite{MR1428513} in the context of vertex reinforcement.

\begin{theorem}
  \label{dcomuni}
  Let $X=(X_t)_{t\geq 0}$ be a cVRRW on $G=K_d$ $(d\geq 3)$. Almost surely, for any $i\in V$, $\lim_{t\to \infty}\pi_i(\mathbf{T}(t))=1/d$.
\end{theorem}
\begin{proof}
  Observe that $\pi(\mathbf{T}(u))\cdot \nabla V(\mathbf{T}(u))=1-\frac{1}{d}-H(u)$. By Proposition \ref{propfTstoapp}, we have
\begin{equation}
  \label{Vstoapp}
  \begin{aligned}
  V(t)-V(0)-\int_0^t (1-\frac{1}{d}-H(u)) d u&=Q \nabla V(\mathbf{T}(t))_{X_t}-Q \nabla V(\mathbf{T}(0))_{X_0}+M_V(t)\\
  &-\int_0^t \frac{\partial}{\partial T_{X_u}} Q \nabla V(\mathbf{T}(u))_{X_u} d u
\end{aligned}
\end{equation}
where we wrote $Q\nabla V (\mathbf{T}):=Q(\mathbf{T}) \nabla V (\mathbf{T})$ for simplicity and $M_V(t)$ is a martingale with quadratic variation 
\begin{equation}
  \label{qudraMV}
  \left\langle M_V\right\rangle_t=\int_0^t \sum_{p: p\sim X_s} W_{p}e^{N_p(\mathbf{T}(s))}([Q(\mathbf{T}(s))\nabla V(\mathbf{T}(s))]_{p}-[Q(\mathbf{T}(s))\nabla V(\mathbf{T}(s))]_{X_s})^2 d s
\end{equation}
Since $\nabla V(\mathbf{T})$ is a bounded vector function with bounded partial derivatives, by (\ref{QestKd}), for some constant $C$, we have 
$$
\max\left\{\left|Q \nabla V(\mathbf{T}(s))_{X_s}\right|, \left|\frac{\partial}{\partial T_{X_s}} Q \nabla V(\mathbf{T}(s))_{X_s}\right|\right\} \leq \frac{C}{\sum_{x \in V} W_x e^{N_x(\mathbf{T}(s))}}
$$
Using this inequality, (\ref{qudraMV}) and the fact that $\sum_{x\in V}W_xe^{N_x(\mathbf{T}(t))}\geq \min_{x\in V}(W_x)e^{t/d}$, we see that the right hand side of (\ref{Vstoapp}) converges a.s.

By the property of $H$ (see e.g. Proposition \ref{Hjconst}), for any small neighborhood $B(z,\delta)$ $(\delta>0)$ of $z:=(\frac{1}{d},\frac{1}{d},\cdots,\frac{1}{d})\in \RR^d$, there exists $\varepsilon>0$ such that $1-\frac{1}{d}-H(\pi)>\varepsilon$ if $\pi$ is outside $B(z,\delta)$. Since $\pi(\mathbf{T}(t))$ is 1-Lipschitz, the time $\pi(\mathbf{T}(t))$ stays outside of $B(z,2\delta)$ must be finite and it leaves $B(z,2\delta)$ only finitely many times, because otherwise 
$$
\lim_{t\to \infty}V(t)= \lim_{t\to \infty}\int_0^{t} (1-\frac{1}{d}-H(u)) d u = \infty
$$
which contradicts the fact $V$ is non-positive. This shows that $\pi(\mathbf{T}(t))$ converges to $z$ a.s.. 
\end{proof}

\begin{corollary}[Section 5.6, \cite{pemantle1988thesis}]
  \label{VRRWKdas}
  Let $(X_n)_{n\in \NN}$ be a VRRW on $G=K_d$ $(d\geq 3)$. Almost surely, for any $x\in V$, $\lim_{n\to \infty}Z_x(n)/n=1/d$.
\end{corollary}
\begin{proof}
By Proposition \ref{vrrwmixcvrrw}, it suffices to prove the result for the discrete-time process associated with a cVRRW. Let $X=(X_t)_{t\geq 0}$ be a cVRRW on $G$ with weights $(W_x)_{x\in V}$ and we denote by $Y_{x}(t)$ the number of visits of $X$ to $x\in V$ up to time $t$. By Corollary \ref{wevexpsum} and the law of large numbers,
\begin{equation}
  \label{llnYW}
  \lim_{t\to \infty}\frac{Y_{x}\left(t\right)}{W_{x}\left(e^{N_{x}(\mathbf{T}(t))}-1\right)} = 1 ,\quad \lim_{t\to \infty}\frac{\sum_{x\in V}Y_{x}\left(t\right)}{\sum_{x\in V} W_{x}\left(e^{N_{x}(\mathbf{T}(t))}-1\right)} = 1 \quad a.s.
\end{equation}
Now apply Theorem \ref{dcomuni} to conclude.
\end{proof}

\begin{proposition}
  Let $G$ be a complete $d$-partite graph with $d\geq 3$ and $X=(X_t)_{t\geq 0}$ be a cVRRW on $G$ with initial weights $(W_x)_{x\in V}$. Then, almost surely, for any $i\in V_p \subset G$, $1\leq p \leq d$, 
  \begin{equation}
      \label{piTdparti}
      \lim_{t\to \infty}\pi_i(\mathbf{T}(t))=\frac{W_i}{\sum_{j\in V_p}W_j}\frac{1}{d}  
  \end{equation}
Moreover, let $(X_n)_{n\in \NN}$ be a VRRW on $G$ with initial local times $(a_x)_{x\in V}$. Then, almost surely, for any $p\in \{1,2,\cdots,d\}$, $\lim_{n\to \infty} (dZ_i(n)/n)_{i\in V_p}$
  exists with distribution $\operatorname{Dirichlet}(a_i,i\in V_p)$
\end{proposition}
\begin{proof}
 Observe that $(\sum_{j\in V_p}T_j(t),t\geq 0)_{1\leq p \leq d}$ is equal in law to the local time process of a cVRRW on $K_d$ with initial weights $(\sum_{j\in V_p}W_j, 1\leq p \leq d)$. Then (\ref{piTdparti}) follows from Theorem \ref{dcomuni}. By results for the P\'{o}lya urn model, for any $p\in \{1,2,\cdots,d\}$,
$$ \lim_{n\to \infty}(\frac{Z_y(n)}{\sum_{j\in V_p} Z_j(n)})_{y \in V_p}\ \text{exists with distribution} \ \operatorname{Dirichlet}(a_y,y\in V_p)$$
Corollary \ref{VRRWKdas} implies that $\sum_{j\in V_p} Z_j(n)/n$ converges to 1/d. Then apply Slutsky's theorem to conclude.
\end{proof}

\subsection{cVRRW on complete-like graphs}
\label{comlikesub} Now we assume that $(X_t)_{t\geq 0}$ is a cVRRW on a complete-like graph $G=S \cup \partial S$. Without loss of generality, we may assume that each $i \in S$ has at most one leaf. Indeed, if $n(j_1)=n(j_2)$, we may glue $j_1,j_2$ together and denote this new vertex by $j$. Then 
$$(X_t,(T_{x}(t),x\neq j_1,j_2),T_{j_1}(t)+T_{j_2}(t))_{t\geq 0} \stackrel{d}{=} (X^{(new)}_t,(T_{x}^{X^{(new)}}(t),x\neq j),T_{j}^{X^{(new)}}(t))_{t\geq 0}$$
where $(X^{(new)}_t)$ is a cVRRW on the new graph with weights $W_j=W_{j_1}+W_{j_2}$ and otherwise unchanged and $T^{X^{(new)}}$ is its local time process. If $i\in S$ has a leaf, we denote the leaf point by $\ell(i)$. 

By Markov property of $A=(A_t)_{t\geq 0}$, we can obtain some bounds on $Q(\mathbf{T})_{x,y}$. 

\begin{lemma}
  \label{mjpcomplike}
  Let $A$ be the Markov chain defined by (\ref{LTdef}). Assume that $W_je^{T_{n(j)}}\leq 2W_{n(j)}e^{N_{n(j)}(\mathbf{T})}$ for all $j\in \partial S$. Then, there exists a constant $C$ such that 
  \begin{enumerate}[topsep=0pt, partopsep=0pt, label=(\roman*)]
\item $\sup_{i \in S}\EE_{i}^{\mathbf{T}}\tau_x \leq \frac{C}{W_{x}e^{N_{x}(\mathbf{T})}}$ for any $x \in S$
\item for  any $y \in \partial S$,
  $$
  \sup_{i \in S}\EE_{i}^{\mathbf{T}}\tau_y \leq \frac{C}{W_{y}e^{N_{y}(\mathbf{T})}}\frac{\sum_{x \in V} W_x e^{N_x(\mathbf{T})}}{W_{n(y)}e^{N_{n(y)}(\mathbf{T})}}
  $$
\item for any $j,y \in \partial S$
  $$
 \EE_{j}^{\mathbf{T}}\tau_y \leq \frac{C}{W_{n(j)}e^{N_{n(j)}(\mathbf{T})}}+\frac{C}{W_{y}e^{N_{y}(\mathbf{T})}}\frac{\sum_{x \in V} W_x e^{N_x(\mathbf{T})}}{W_{n(y)}e^{N_{n(y)}(\mathbf{T})}}
  $$
  \item  $$
 \sup_{x\in S,y\in S}Q(\mathbf{T})_{x,y} \leq \frac{C}{\sum_{x \in V} W_x e^{N_x(\mathbf{T})}}, \quad \sup_{x\in S}Q(\mathbf{T})_{x,y} \leq \frac{C}{W_{n(y)}e^{N_{n(y)}(\mathbf{T})}}, \ \forall y\in \partial S
$$
Moreover,
$$
\sup_{x\in V,  y\in V} Q(\mathbf{T})_{x,y} \leq \frac{C}{\min_{i \in S} W_i e^{N_i(\mathbf{T})}}
$$
and 
\begin{equation}
    \label{Qxyjdifbd}
    \sup_{x\in S,  y\in S, j\in V} |Q(\mathbf{T})_{x,j}-Q(\mathbf{T})_{y,j}| \leq \frac{C}{\sum_{i \in V} W_i e^{N_i(\mathbf{T})}}; \  \sup_{j\in V} |Q(\mathbf{T})_{x,j}-Q(\mathbf{T})_{n_x,j}| \leq \frac{C}{ W_{n_x} e^{N_{n_x}(\mathbf{T})}}, \ \forall x\in \partial S
\end{equation}
  \end{enumerate}  
\end{lemma}
\begin{proof}
(i) Before $A$ hit $x$, the expected number of visits to any $j\in \partial S$ is bounded by $W_je^{N_j(\mathbf{T})}/W_{x}e^{N_{x}(\mathbf{T})}$ whence the expected time spent at $j$ is bounded by 
$$\frac{W_je^{N_j(\mathbf{T})}}{W_{n(j)}e^{N_{n(j)}(\mathbf{T})}W_{x}e^{N_{x}(\mathbf{T})}} \leq \frac{2}{W_{x}e^{N_{x}(\mathbf{T})}}$$ 
By definition, the expected time spent on $S$ before $\tau_x$ is $(W_{x}e^{N_{x}(\mathbf{T})})^{-1}$. Thus, 
  $$
  \EE_{i}^{\mathbf{T}}\tau_{x} \leq \frac{1}{W_{x}e^{N_{x}(\mathbf{T})}}+\frac{2(d-1)}{W_{x}e^{N_{x}(\mathbf{T})}}
  $$
(ii) By the Markov property, $\EE_{i}\tau_{y}=\EE_{i}\tau_{n(y)}+\EE_{n(y)}\tau_y$. By (i) and similar arguments,
  $$
\EE_{n(y)}^{\mathbf{T}}\tau_y\leq \frac{1}{W_ye^{N_y(\mathbf{T})}}+\frac{\sum_{x\in S:x\sim n_y}W_{x}e^{N_{x}(\mathbf{T})}}{W_ye^{N_y(\mathbf{T})}}\frac{C}{W_{n(y)}e^{N_{n(y)}(\mathbf{T})}}
  $$
 which implies (ii).  \\
 (iii) For any $j \in \partial S$ and $j\neq y$, $\EE_{j}^{\mathbf{T}}\tau_y=\EE_{n(j)}^{\mathbf{T}}\tau_y+\frac{1}{W_{n(j)}e^{N_{n(j)}(\mathbf{T})}}$. (iii) then follows from (ii). \\
 (iv) follows from (i), (ii), (iii), and (\ref{QTij}). Note that for (\ref{Qxyjdifbd}), one can use (\ref{Qijexpress}) to write 
 $$
Q(\mathbf{T})_{x,j}-Q(\mathbf{T})_{y,j}=\pi_j (\EE_x^{\mathbf{T}}\tau_j-\EE_y^{\mathbf{T}}\tau_j)
 $$
 Then the first inequality in (\ref{Qxyjdifbd}) follows from (i) if $j \in S$. If $j\in \partial S$, $\EE_x^{\mathbf{T}}\tau_j-\EE_y^{\mathbf{T}}\tau_j=\EE_x^{\mathbf{T}}\tau_{n_j}-\EE_y^{\mathbf{T}}\tau_{n_j}$ by the Markov property. The second inequality in (\ref{Qxyjdifbd}) is proved similarly.
\end{proof}

Now we are ready to prove Theorem \ref{cvrrwcomlike}.

\begin{proof}[Proof of Theorem \ref{cvrrwcomlike}]
 As in Corollary \ref{VRRWKdas}, fix $j \in \partial S$, by the law of large numbers for Poisson process, almost surely,
 \begin{equation}
  \label{NnjTandTnj}
  \liminf_{t\to \infty}\frac{W_{n_j}e^{N_{n_j}(T(t))}}{W_je^{T_{n_j}(t)}} \geq 1, \quad \text{i.e.}\ \liminf_{t\to \infty}N_{n_j}(\mathbf{T}(t)) - T_{n_j} \geq \log \frac{W_j}{W_{n_j}}
 \end{equation}
 Denote the $n$-th visit time to $n_j$ by $\sigma_n$. By the time-lines construction(Proposition \ref{proptimeline1}), 
 \begin{equation}
  \label{Tjbdfinite}
   T_j(\sigma_{k+1})\leq \Theta_1^{(n_j)}+\sum_{n=1}^{k}\left(\Theta_{n+1}^{(n_j)}-\Theta_{n}^{(n_j)}\right) \mathds{1}_{B_n}
 \end{equation}
 where $B_n$ is the event that the next jump after $\sigma_n$ is $n_j \to j$. Similarly, 
$$
\sum_{p\in S:p\sim n_j}T_p(\sigma_{k+1}) \geq \sum_{n=1}^{k}\left(\Theta_{n+1}^{(n_j)}-\Theta_{n}^{(n_j)}\right) (1-\mathds{1}_{B_n})
$$
Note that 
$$\PP\left(B_n|\FF_{\sigma_n}\right)=\frac{W_je^{T_{n_j}(\sigma_n)}}{\sum_{x\in S: x\sim n_j}W_x e^{N_x(\mathbf{T}(\sigma_n))}}\leq \frac{W_j}{\sum_{x\in S: x\sim n_j}W_x}$$
Martingale arguments then imply that 
$$\limsup_{t\to\infty}\frac{T_j(t)}{\sum_{p\in S:p\sim n_j}T_p(t)} \leq \frac{W_j}{\sum_{x\in S: x\sim n_j}W_x}$$
  Combined with (\ref{NnjTandTnj}), this implies that for large $t$, we can always find some $p \in S$, with $p\sim n_j$ such that 
  \begin{equation}
  \label{Tsbdc1}
      \frac{T_p(t)}{T_{n_j}(t)} > \frac{\sum_{x\in S: x\sim n_j}W_x}{d \sum_{x: x\sim n_j}W_x}=:c_1
  \end{equation}
  Note that we used the assumption that $d\geq 3$ here. In particular,
 \begin{equation}
  \label{Wiest}
  \limsup_{t\to \infty} \frac{(W_je^{T_{n_j}(t)})^{c_1+1}}{\sum_{x\in S: x\sim n_j}W_x e^{N_x(\mathbf{T}(t))}} =0, \quad \text{and thus} \quad \int_0^{\infty}\pi_j(\mathbf{T}(t))dt<\infty\quad   a.s.
 \end{equation}

By Proposition \ref{propeFtwi},
 $$
  \sum_{n=1}^{\infty} \EE\left(\left(\Theta_{n+1}^{(i)}-\Theta_{n}^{(i)}\right) \mathds{1}_{B_n}|\FF_{\sigma_n}\right)\leq  \sum_{n=1}^{\infty} \frac{W_je^{T_{n_j}(\sigma_n)}}{\sum_{x\in S: x\sim n_j}W_xe^{N_x(\mathbf{T}(\sigma_n))}}\frac{1}{W_{n_j}e^{N_{n_j}(\mathbf{T}(\sigma_n))}}
 $$
 By (\ref{Wiest}) and the law of large numbers, for large $n$, $W_{n_j}e^{N_{n_j}(\mathbf{T}(\sigma_n))}\leq 2\sum_{x\in S: x\sim n_j}W_xe^{N_x(\mathbf{T}(\sigma_n))}$ and thus for some constant $C$, we have
 $$
 \frac{W_je^{T_{n_j}(\sigma_n)}}{\sum_{x\in S: x\sim n_j}W_xe^{N_x(\mathbf{T}(\sigma_n))}}\frac{1}{W_{n_j}e^{N_{n_j}(\mathbf{T}(\sigma_n))}} \leq \frac{C}{(W_{n_j}e^{N_{n_j}(\mathbf{T}(\sigma_n))})^{2-\frac{1}{c+1}}}\sim \frac{C}{n^{2-\frac{1}{c_1+1}}}
 $$
 which is summable. Therefore, almost surely,
 \begin{equation}
  \label{condexsummable}
  \sum_{n=1}^{\infty} \EE\left(\left(\Theta_{n+1}^{(i)}-\Theta_{n}^{(i)}\right) \mathds{1}_{B_n}|\FF_{\sigma_n}\right)<\infty
 \end{equation}
 Again, by Proposition \ref{propeFtwi} and Corollary \ref{wevexpsum}, letting $\xi_1,\xi_2,\cdots$ be i.i.d. Exp(1)-distributed random variables, we have,
 $$
\EE\left(\Theta_{n+1}^{(n_j)}-\Theta_{n}^{(n_j)}\right)^2 \leq \EE\frac{2}{(W_{n_j}e^{\Theta_{n}^{(n_j)}})^2}=\EE\frac{2}{(W_{n_j}+\sum_{i=1}^n\xi_i)^2}
 $$
 Since $\sum_{i=1}^{n} \xi_i \sim \operatorname{Gamma}(n,1)$, for $n\geq 3$, we have
$$
\begin{aligned}
  \EE\frac{2}{(W_{n_j}+\sum_{i=1}^n\xi_i)^2}=\int_0^{\infty}\frac{2}{(W_{n_j}+x)^2}\frac{x^{n-1} \mathrm{e}^{-x}}{(n-1) !}dx&\leq \frac{2}{(n-1)(n-2)}\int_0^{\infty}\frac{x^{n-3} \mathrm{e}^{-x}}{(n-3) !}dx\\
  &=\frac{2}{(n-1)(n-2)}
\end{aligned}
$$
which is summable. Thus, a.s. the martingale 
$$M_{n+1}:=\sum_{i=1}^{n}\left(\left(\Theta_{n+1}^{(n_j)}-\Theta_{n}^{(n_j)}\right) \mathds{1}_{B_n}-\EE(\left(\Theta_{n+1}^{(n_j)}-\Theta_{n}^{(n_j)}\right) \mathds{1}_{B_n}|\FF_{\sigma_n}) \right),\quad n\geq 1$$
converges. By (\ref{Tjbdfinite}) and (\ref{condexsummable}), we have a.s. 
\begin{equation}
  \label{Tjpartialfinite}
  T_j(\infty)<\infty, \quad \forall j\in \partial S
\end{equation}

Again, we define 
$$V(t):=V(\mathbf{T}(t)):=\frac{1}{d}\sum_{i\in S}\log \pi_i(\mathbf{T}(t))=-\log \sum_{x\in V}W_xe^{N_x(\mathbf{T}(t))}+\frac{1}{d}\sum_{i\in S}\left[N_i(\mathbf{T}(t))+\log W_i\right] $$
By Proposition \ref{propfTstoapp}, we have
\begin{equation}
  \label{VKdlikestoapp}
  \begin{aligned}
   V(t)-V(0)-\int_0^t (1-\frac{1}{d}-H(u)) &d u=Q \nabla V(\mathbf{T}(t))_{X_t}-Q \nabla V(\mathbf{T}(0))_{X_0}+M_V(t)\\
  &-\int_0^t \frac{\partial}{\partial T_{X_u}} Q \nabla V(\mathbf{T}(u))_{X_u} d u+\sum_{j\in \partial S}\int_0^{t}C_j\pi_j(\mathbf{T}(u))du
\end{aligned}
\end{equation}
where $C_j$ are constants depending on $G$ and the quadratic variation of $M_V$ is given by (\ref{qudraMV}). Note that $\nabla V$ is a bounded column-vector function with bounded partial derivatives.

By (\ref{NnjTandTnj}), (\ref{Tsbdc1}) and (\ref{Tjpartialfinite}), a.s. we can find some (possibly random) $u>0$ such that for all $t\geq u$,
\begin{equation}
    \label{goodu}
  \min_{x \in S} W_x e^{N_x(\mathbf{T})} \geq e^{c_2t}; \quad W_je^{T_i(t)}\leq 2W_{n_j}e^{N_{n_j}(T(t))}, \quad \forall j\in \partial S  
\end{equation}
where $c_2$ is a positive constant. By (\ref{qudraMV}), (\ref{Qxyjdifbd}) and (\ref{goodu}), we have
\begin{equation}
  \label{MVinftybd}
  \left\langle M_V\right\rangle_{\infty}\leq \left\langle M_V\right\rangle_{u}+\int_u^{\infty} \frac{C }{\min_{x \in S} W_x e^{N_x(\mathbf{T}(t))}}d t<\infty
\end{equation}

For all $k, l \in V$, let us compute $\frac{\partial}{\partial T_{k}} Q(\mathbf{T})_{k, l}$: by differentiation of the Poisson equation,
 \begin{equation}
   \label{deriQT}
    \begin{aligned}
 &\quad\ \frac{\partial}{\partial T_{k}} Q(\mathbf{T})_{k, l}=\left(Q(\mathbf{T})L(\mathbf{T})\frac{\partial}{\partial T_{k}} Q(\mathbf{T})+\Pi(\mathbf{T})\frac{\partial}{\partial T_{k}} Q(\mathbf{T})\right)_{k, l} \\  
 &=-\left(Q(\mathbf{T})\frac{\partial}{\partial T_{k}} \Pi(\mathbf{T})+ Q(\mathbf{T})\left(\frac{\partial}{\partial T_{k}} L\right) Q(\mathbf{T})\right)_{k, l} -  \left((\frac{\partial}{\partial T_{k}}\Pi(\mathbf{T}))  Q(\mathbf{T})\right)_{k, l} \\
 &=-\left( Q(\mathbf{T})\left(\frac{\partial}{\partial T_{k}} L\right) Q(\mathbf{T})\right)_{k, l}- (\widetilde{\Pi}_k(\mathbf{T})Q(\mathbf{T}))_{k, l}
 \end{aligned}
 \end{equation}
 where $\widetilde{\Pi}_k(\mathbf{T})(u,v)=\pi_v(\mathbf{T})$ if $v \sim k$ and equals 0 otherwise. Note that we used
 $$
 \frac{\partial}{\partial T_{k}}\Pi(\mathbf{T})=\widetilde{\Pi}_k(\mathbf{T})-\frac{Z_k(\mathbf{T})}{\sum_{i \in V} W_ie^{N_i(\mathbf{T})}}\Pi(\mathbf{T}), \quad \text{and} \quad Q(\mathbf{T})\widetilde{\Pi}_i(\mathbf{T})=0
 $$
 where $Z_k(\mathbf{T}):=\sum_{y: y \sim k} W_y e^{N_y(\mathbf{T})}$. On the other hand,
 $$
 (\frac{\partial}{\partial T_{k}} L(\mathbf{T}))_{r,q} = \begin{cases}-\sum_{p: p\sim r,p\sim k} W_pe^{N_p(\mathbf{T})} & \text { if } r= q \\ W_qe^{N_q(\mathbf{T})} & \text { if } q \sim k, r \sim q \\ 0 & \text { otherwise }\end{cases}
 $$
 Therefore, if $t\geq u$, by Lemma \ref{mjpcomplike} (iv) and (\ref{deriQT}),  considering the two cases $X_t\in S$ and $X_t\notin S$ respectively, one can show that, for some constant $C$,
 \begin{equation}
  \label{derQbetabd}
  \left|\frac{\partial}{\partial T_{X_t}} Q \nabla V(\mathbf{T}(t))_{X_t}\right| \leq \frac{C}{\min_{x \in S} W_x e^{N_x(\mathbf{T}(t))}}
 \end{equation}
Combining (\ref{Wiest}), Lemma \ref{mjpcomplike} (iv), (\ref{MVinftybd}) and (\ref{derQbetabd}), we know that (\ref{VKdlikestoapp}) converges a.s. as $t\to \infty$.

Arguments as in the proof of Theorem \ref{dcomuni} then imply that almost surely $\pi(\mathbf{T}(t))$ converges to $z^*$ defined in (\ref{zstardef}).

By (\ref{Tjpartialfinite}), a.s. $\sum_{i\in S}T_i(t)-t$ converges. For any $i,j\in S$, since $\pi_i(\mathbf{T}(t))/\pi_j(\mathbf{T}(t))\to 1$, $T_j(t)-T_i(t)$ converges a.s. Thus, for any $i\in S$, $T_i(t)-t/d$ converges a.s.
\end{proof}

In the proof of the a.s-convergence of (\ref{VKdlikestoapp}), we used only the fact that $\nabla V$ is a bounded column-vector function with bounded partial derivatives. Therefore, when $G$ is a complete-like graph, we improved Proposition \ref{propfTstoapp} in the following sense.

\begin{theorem}[Stochastic approximation]
 \label{fTasconv}
 In the setting of Proposition \ref{propfTstoapp}, if $G$ is a complete-like graph and $f$ has bounded partial derivatives and bounded second partial derivatives, then almost surely, $M_{f}(t)$ and $\int_{0}^{t} \frac{\partial}{\partial T_{X_{u}}} Q\nabla f(\mathbf{T}(u))_{X_{u}} d u$ converge, and
$$
\begin{aligned}
  &\quad \lim_{t\to \infty}\left(f(\mathbf{T}(t))-f(\mathbf{T}(0))-\int_0^t  \pi(\mathbf{T}(u))\cdot \nabla f(\mathbf{T}(u)) d u \right)\\
  &=-Q\nabla f(\mathbf{T}(0))_{X_{0}}+\lim_{t \to \infty}M_{f}(t)-\int_{0}^{\infty} \frac{\partial}{\partial T_{X_{u}}} Q\nabla f(\mathbf{T}(u))_{X_{u}} d u
\end{aligned}
$$

\end{theorem}

As in Corollary \ref{VRRWKdas}, we can deduce Corollary \ref{VRRWcomlikeas} from Theorem \ref{cvrrwcomlike}

\begin{proof}[Proof of Corollary \ref{VRRWcomlikeas}]
  The proof of (i) is similar to that of Corollary \ref{VRRWKdas}. (ii) First note that by Theorem \ref{cvrrwcomlike}, 
  \begin{equation}
    \label{Njtdconv}
      \frac{W_je^{N_j(\mathbf{T}(t))}}{W_je^{t/d}}=e^{T_{n_j}(t)-t/d} \quad \text{and}\quad \frac{\sum_{x\in V}W_xe^{N_x(\mathbf{T}(t))}}{\sum_{x\in V}W_xe^{(1-\frac{1}{d})t}}=\sum_{x\in V}W_xe^{N_x(\mathbf{T}(t))-(1-\frac{1}{d})t}
  \end{equation}
  converge a.s. to non-zero limits. Let $(Y_x(t), t\geq 0)_{x\in V}$ be as in the proof of Corollary \ref{VRRWKdas}. Then by the law of large numbers, almost surely, for any $j\in \partial S$,
  $$
\lim_{t\to \infty}\frac{Y_{j}(t)}{(\sum_{x\in V}Y_x(t))^{\frac{1}{d-1}}}= \lim_{t\to \infty}\frac{W_je^{N_j(\mathbf{T}(t))}}{(\sum_{x\in V}W_xe^{N_x(\mathbf{T}(t))})^{\frac{1}{d-1}}}  \quad \text{ exists and is non-zero}
$$
It remains to apply Proposition \ref{vrrwmixcvrrw}.
\end{proof}

\begin{proof}[Proof of Theorem \ref{comdpartthm}]
  The proof of (i) is similar to that of (i) in Theorem \ref{cvrrwcomlike}. In particular, for any $i \in V_p \subset S$,
  $$
  \lim_{t\to \infty} \frac{\pi_i(\mathbf{T}(t))}{\sum_{y\in V_p}\pi_y(\mathbf{T}(t))}
  $$
  exists and is non-zero. For (ii), define 
  $$V(t):=-\log \sum_{x\in V}W_xe^{N_x(\mathbf{T}(t))} +\frac{1}{d}\sum_{q=1}^d \left(\sum_{x \in S, x \notin V_q}T_x(t)+\sum_{y\in V_q}\log W_y\right) $$
  Then $V(t)$ is upper bounded by $
\sum_{q=1}^d \frac{1}{d} \log (\sum_{y\in V_q} \pi_y(\mathbf{T}(t)))\leq 0$.
  Then (\ref{VKdlikestoapp}) still holds with possibly different constants $(C_j)_{j\in \partial S}$. Similar arguments as in the proof of Theorem \ref{cvrrwcomlike} imply that $\sum_{y\in V_p}\pi_y(\mathbf{T}(t))$ converges to $1/d$ a.s. for any $p\in \{1,2,\cdots,d\}$. 
  
  (I) follows from (ii) and Proposition \ref{vrrwmixcvrrw} as in the proof of Corollary \ref{VRRWKdas}. For (II), by Proposition \ref{propfTstoapp} with $f(\mathbf{T})=T_{n_j}$, as in Theorem \ref{fTasconv}, we have
  $$
  T_{n_j}(t)-\int_0^t\pi_{n_j}(\mathbf{T}(t))dt  \quad \text{converges a.s.}
  $$
  Therefore, as in the proof of Corollary \ref{VRRWKdas}, let $Y_{x}(t)$ the number of visits of $(X_t)_{t\geq 0}$ to $x\in V$ in the time interval $(0,t]$, then
  $$
  \lim_{t \to \infty} \frac{\log W_{j}e^{T_{n_j}(t)}}{\log \sum_{x\in V}W_xe^{N_x(\mathbf{T}(t))}} =   \frac{d }{d-1} z_{n_j} =  \frac{d }{d-1} \lim_{t\to \infty} \frac{Y_{n_j}(t)}{\sum_{x\in V} Y_x(t)}
  $$
  It remains to apply Proposition \ref{vrrwmixcvrrw}.
\end{proof}

\subsection{On the rate of convergence}
\label{rateconsub}

\begin{proof}[Proof of Theorem \ref{rateupper}]
  We first assume that $d>3$. Let $i,j$ be two distinct vertices in $S$. Recall that $\ell(i)$ and $\ell(j)$ are their leaves (if any). Applying Proposition \ref{propfTstoapp} with $f(\mathbf{T})=W_{i}e^{T_j+T_{\ell(i)}}-W_je^{T_i+T_{\ell(j)}}$, we have
\begin{equation}
  \label{piipijdiff}
    \begin{aligned}
      f(\mathbf{T}(t))=&= W_i-W_j+[Q(\mathbf{T}(t))\nabla f(\mathbf{T}(t))]_{X_{t}}-[Q(\mathbf{T}(0))\nabla f(\mathbf{T}(0))]_{X_{0}} \\
    &+\int_0^t \frac{\left( W_{\ell(i)}W_ie^{T_i(u)+T_j(u)+T_{\ell(i)}(u)}-W_{\ell(j)}W_je^{T_i(u)+T_j(u)+T_{\ell(j)}(u)}\right)}{\sum_{x\in V} W_xe^{N_x(\mathbf{T})}}   du  \\
      &-\int_{0}^{t} \frac{\partial}{\partial T_{X_{u}}} Q\nabla f(\mathbf{T}(u))_{X_{u}} d u+M_{f}(t)
      \end{aligned}
\end{equation}
By Lemma \ref{mjpcomplike}, (\ref{deriQT}) (or (\ref{derQbetabd})) and Theorem \ref{cvrrwcomlike}, if $d>3$, the right hand side of (\ref{piipijdiff}) converges a.s., where we used 
$$
 \limsup_{t\to \infty}e^{\frac{t}{d}} g_f(X_t,\mathbf{T}(t))<\infty
$$
to show that $M_f(t)$ converges a.s. Therefore, by Theorem \ref{cvrrwcomlike}, almost surely,
\begin{equation}
  \label{piipijetd}
  \lim_{t\to \infty} e^{\frac{t}{d}} \left( \pi_i(\mathbf{T}(t))-\pi_j(\mathbf{T}(t))\right)= \lim_{t\to \infty}\frac{e^{\frac{t}{d}}e^{\sum_{y\in S: y\sim i, y\sim j}T_y(t)}}{\sum_{x\in V} W_xe^{N_x(\mathbf{T}(t))}}\left( W_1e^{T_3(t)}-W_3e^{T_1(t)}\right) \quad \text{exists}
\end{equation}
Note that $e^{\frac{t}{d}}\pi_{\ell}(\mathbf{T}(t))$ converges to 0 a.s. for any leaf $\ell \in \partial S$ and $\sum_{x\in V}\pi_x(\mathbf{T}(t))=1$. Then, by (\ref{piipijetd}),
$$
\lim_{t\to \infty} e^{\frac{t}{d}}( \pi_i(\mathbf{T}(t))-\frac{1}{d} ) \quad \text{exists a.s. for any}\ i\in S
$$
which proves the existence of the first limit in (\ref{convratedlarge3}). By (\ref{piipijdiff}), almost surely,
$$
\lim_{t\to\infty}\left(\frac{e^{T_i(t)-T_{\ell(i)}(t)}}{W_i}-\frac{e^{\frac{t-2\sum_{x\in \partial S}T_x(t)}{d}}}{(\prod_{j\in S} W_j)^{\frac{1}{d}}} \right)=\lim_{t\to\infty}\left(\frac{e^{T_i(t)-T_{\ell(i)}(t)}}{W_i}-\sqrt[d]{\prod_{j\in S}\frac{e^{T_j(t)-T_{\ell(j)}(t)}}{W_j}}\right) \quad \text{exists}
$$
which completes the proof of (\ref{limTiminustd}).

If $(U(t))_{t\geq 0}$ be a Poisson process with unit rate, then it is known that $(M_U(t):=U(t)- t)_{t\geq 0}$ is a martingale with $\langle M_U\rangle_t=t$. Since $M_U$ has bounded jumps, by the law of the iterated logarithm for right-continuous local martingales, see e.g. Theorem 3 in \cite{MR0520004}, a.s.
\begin{equation}
 \label{lilMu}
\limsup_{t \to \infty}\frac{M_{U}(t)}{\sqrt{2t\log \log t}} \leq 1
\end{equation}
We denote by $Y_{x}(t)$ the number of visits of $(X_t)_{t\geq 0}$ to $x\in V$ up to time $t$. By (\ref{piipijetd}), (\ref{lilMu}), Corollary \ref{wevexpsum} and Theorem \ref{cvrrwcomlike}, for any $i\neq j\in S$, almost surely,
$$
\lim_{t\to \infty}(Y_i(t))^{\frac{1}{d-1}} \log \frac{Y_j(t)}{Y_i(t)}= \lim_{t\to \infty}(W_ie^{N_i(\mathbf{T}(t))})^{\frac{1}{d-1}} \left(\log \frac{Y_j(t)}{W_je^{N_j(\mathbf{T}(t))}} + \log  \frac{W_ie^{N_i(\mathbf{T}(t))}}{Y_i(t)}+\log \frac{W_je^{N_j(\mathbf{T}(t))}}{W_ie^{N_i(\mathbf{T}(t))}}  \right) 
$$
exists, where we used $\log (1+x) \sim x$ as $x\to 0$. Let $\eta_{n}$ be the $n$-th jump time of $X$. Then, 
\begin{equation}
  \label{Yeta1d-1log}
  \lim_{n \to \infty}(Y_i(\eta_n))^{\frac{1}{d-1}} \log \frac{Y_j(\eta_n)}{Y_i(\eta_n)}\quad \text{exists a.s.}
\end{equation}
By Proposition \ref{vrrwmixcvrrw}, for the VRRW $(X_n)_{n\in \NN}$, 
$$
\lim_{n \to \infty}(Z_i(n))^{\frac{1}{d-1}} \log \frac{Z_j(n)}{Z_i(n)}\quad \text{exists a.s.}
$$
By Corollary \ref{VRRWcomlikeas}, this implies the existence of the second limit in (\ref{convratedlarge3}).

The case $d=3$ is proved similarly. Simply observe that by (\ref{piipijdiff}), for any $\varepsilon>0$ and $i\neq j \in S$, almost surely, $e^{-\varepsilon t}\left( W_{i}e^{T_j+T_{\ell(i)}}-W_je^{T_i+T_{\ell(j)}}\right)$ converges to 0.
\end{proof}

\begin{proof}[Proof of Proposition \ref{K3lb}]
  We label the three vertices by $1,2,3$ and let $f(\mathbf{T}):=W_1e^{T_3}-W_3e^{T_1}$. By (\ref{piipijdiff}),
$$
    \begin{aligned}
    &\quad W_1e^{T_3(t)}-W_3e^{T_1(t)}=W_1-W_3+[Q(\mathbf{T}(t))\nabla f(\mathbf{T}(t))]_{X_{t}}-[Q(\mathbf{T}(0))\nabla f(\mathbf{T}(0))]_{X_{0}} \\
      &-\int_{0}^{t} \frac{\partial}{\partial T_{X_{u}}} Q\nabla f(\mathbf{T}(u))_{X_{u}} d u+M_{f}(t)
      \end{aligned}
$$
By Lemma \ref{QKd}, 
\begin{equation}
  \label{QnablafK3}
  [Q\nabla f(\mathbf{T})]_{1}=\frac{W_3e^{T_1}}{\sum_{i=1}^3 W_ie^{N_i(\mathbf{T})}},\quad [Q\nabla f(\mathbf{T})]_{2}=0,\quad [Q\nabla f(\mathbf{T})]_{3}=\frac{-W_1e^{T_3}}{\sum_{i=1}^3 W_ie^{N_i(\mathbf{T})}}
\end{equation}
and thus by Theorem \ref{cvrrwcomlike}, a.s. 
$$
\lim_{t\to\infty}[Q(\mathbf{T}(t))\nabla f(\mathbf{T}(t))]_{X_{t}}=0, \quad \int_{0}^{\infty} |\frac{\partial}{\partial T_{X_{u}}} Q\nabla f(\mathbf{T}(u))_{X_{u}} |d u<\infty
$$
and 
$$
 \limsup_{t\to\infty}g_f(X_t,T(t)) <\infty
$$
where $g_f$ is defined in (\ref{defgxtgeneral}). Therefore, as in (\ref{lilMu}), by the law of the iterated logarithm for martingales, for any $\kappa>1/2$, almost surely,
\begin{equation}
  \label{lilmartW1w3}
  \limsup_{t\to \infty} \frac{M_f(t)}{t^{\kappa}}=0, \quad \text{and thus,} \ \limsup_{t\to \infty} \frac{ W_1e^{T_3(t)}-W_3e^{T_1(t)}}{t^{\kappa}}=0
\end{equation}
By Theorem \ref{cvrrwcomlike}, almost surely,
$$
\limsup_{t\to \infty} \frac{e^{\frac{t}{3}} \left( \pi_1(\mathbf{T}(t))-\pi_3(\mathbf{T}(t))\right)}{t^{\kappa}} = \limsup_{t\to \infty}\frac{e^{\frac{t}{3}}e^{T_2(t)}}{\sum_{i=1}^3 W_ie^{N_i(\mathbf{T}(t))}} \frac{ \left( W_1e^{T_3(t)}-W_3e^{T_1(t)}\right)}{t^{\kappa} }=0
$$
Thus, for any $i\neq j \in \{1,2,3\}$, $e^{\frac{t}{3}}t^{-\kappa}|\pi_i(\mathbf{T}(t))-\pi_j(\mathbf{T}(t))|$ converges to 0, which completes the proof of the first equality in (\ref{limsup3cVRRW}) since 
$$
\max_{i=1,2,3}\{|\pi_i(\mathbf{T}(t))-\frac{1}{3}|\} \leq \max_{i\neq j}\{|\pi_i(\mathbf{T}(t))-\pi_j(\mathbf{T}(t))|\}
$$
Now we prove the second equality in (\ref{limsup3cVRRW}). We consider the optional quadratic variation $[M]_t$ of $M_f$, for a definition of the optional quadratic variation, see e.g. Chapter 11 in \cite{MR3443368}. Since $[M]_t\geq \sum_{0<s \leq t} \Delta M_s^2$, by considering all the jumps to site 2 before time $t$ and (\ref{QnablafK3}), we have
$$
\liminf_{t\to\infty}\frac{[M]_t}{t} >0, \quad \limsup_{t\to\infty}\frac{\left( W_1e^{T_3(t)}-W_3e^{T_1(t)}\right)}{\sqrt{t}}=\infty
$$
by the law of the iterated logarithm for martingales, see e.g. Theorem 4 in \cite{MR0520004}. Then, by Theorem \ref{cvrrwcomlike},
\begin{equation}
  \label{pip3sqrtinf}
  \limsup_{t\to\infty}\frac{W_1e^{N_1(\mathbf{T}(t))}-W_3e^{N_3(\mathbf{T}(t))}}{\sqrt{t W_1e^{N_1(\mathbf{T}(t))}}}=\limsup_{t\to\infty}\frac{e^{T_2(t)}}{\sqrt{W_1e^{N_1(\mathbf{T}(t))}}}\frac{\left( W_1e^{T_3(t)}-W_3e^{T_1(t)}\right)}{\sqrt{t}}=\infty
\end{equation}
which implies the second equality in (\ref{limsup3cVRRW}).

Again, let $Y_{i}(t)$ be the number of visits of the cVRRW $(X_t)_{t\geq 0}$ to $i$ up to time $t$. Since $(Y_i(t))$ changes only at jumping times $\eta_1,\eta_2,\cdots$, as in (\ref{Yeta1d-1log}), by (\ref{lilMu}) and (\ref{pip3sqrtinf}), almost surely,
$$
\limsup_{t\to \infty}\frac{\sqrt{Y_3(t)}}{\sqrt{\log Y_3(t)}}\log \frac{Y_1(t)}{Y_3(t)}=\infty, \quad \limsup_{n\to \infty}\frac{\sqrt{Y_3(\eta_n)}}{\sqrt{\log Y_3(\eta_n)}}\log \frac{Y_1(\eta_n)}{Y_3(\eta_n)}=\infty
$$
Proposition \ref{vrrwmixcvrrw} then implies that a.s.
$$
\limsup_{n\to \infty}\frac{\sqrt{n}}{\sqrt{\log n}}\log \frac{z_1(n)}{z_3(n)}=\infty,  \quad \text{and thus,} \ \limsup_{n\to \infty}\frac{\sqrt{n}}{\sqrt{\log n}}(z_1(n)-z_3(n))=\infty
$$
which completes the proof of the second equality in (\ref{limsup3VRRW}). The first one is proved similarly.
\end{proof}

For any probability measure $x=\left(x_i\right)_{i \subset V}$ on $V$, let
$$
J(x):=2 \sum_{i \in V} x_i\left(N_i(x)-H(x)\right)^2
$$
Recall $z^*$ in (\ref{zstardef}). By the definitions of $H, J$ and some direct computations, we have the following results. The proof is omitted here. 
\begin{proposition}
  \label{Hjconst}
 Let $(x_i)_{i\in V}$ be a probability measure on $G=K_d$ $(d\geq 3)$, then $H(z^*)-H(x)=\sum_{i\in V}(x_i-\frac{1}{d})^2$ and 
  $$
  J(x) = 2\sum_{i\in V}x_i(x_i-\frac{1}{d})^2-2(H(z^*)-H(x))^2
  $$
  In particular, for any $C_1<2/d<C_2$, we can find a neighborhood $B(z^*,\delta)$ of $z^*$ such that 
  $$
  C_1 (H(z^*)-H(x))\leq J(x) \leq C_2 (H(z^*)-H(x))
  $$
\end{proposition}

The definition of $J(x)$ is motivated by the following observation: Informally, by Theorem \ref{fTasconv}, $\mathbf{T}(t)$ is “approximately" governed by $\mathbf{T}^{\prime}=\pi(\mathbf{T})$, then $H(t):=H(\pi(\mathbf{T}(t)))$ is governed by 
$$\frac{d}{dt} (H(z^*)-H(t)) =-2\sum \pi_i(\mathbf{T})\left(N_i(\pi(\mathbf{T}))-H(\pi(\mathbf{T}))\right)^2\approx -\frac{2}{d}(H(z^*)-H(t))$$
when $\pi(\mathbf{T}(t))$ is close to $z$ by Proposition \ref{Hjconst}. This would allow us to estimate to rate of convergence $H(z^*)-H(t) \to 0$.

 We write $J(t):=J(\pi(\mathbf{T}(t)))$ for simplicity, and define $D(t):=\log \sum_{x\in V}W_xe^{N_x(\mathbf{T}(t))}$. With a slight abuse of notation, we let $C$ denote a universal positive constant, and let $C(\omega)$ denote a positive constant depending on the outcome $\omega$ but not on time $t$.

\begin{proof}[Proof of Theorem \ref{d-1optimal}]
By Theorem \ref{fTasconv} with $f(\mathbf{T}):=H(\pi(\mathbf{T}))$,
\begin{equation}
  \label{stoappequHinft}
  \begin{aligned}
    H(z^*)-H(t)&=\int_t^\infty J(u)du -Q\nabla H\circ \pi (\mathbf{T}(t))_{X_{t}}+M_{H\circ \pi}(\infty)-M_{H\circ \pi}(t)\\
      &-\int_{t}^{\infty} \frac{\partial}{\partial T_{X_{u}}} Q\nabla H\circ \pi(\mathbf{T}(u))_{X_{u}} d u 
  \end{aligned}
\end{equation}
By Proposition \ref{Hjconst}, for any $i\in V$,
  $$
  \partial_{T_i} H(\pi(\mathbf{T}))=\pi_i(\mathbf{T})\left(H(\pi(\mathbf{T}))-H\left(z^*\right)+\pi_i(\mathbf{T})-\frac{1}{d}\right)
  $$
Therefore, by Lemma \ref{QKd} and Proposition \ref{Hjconst}, for some constant $C$,
\begin{equation}
  \label{gHbdCh}
  g_{H \circ \pi}\left(X_u, T(u)\right) \leq \frac{C\left(H\left(z^*\right)-H(u)\right)}{\sum_{x \in V} W_x e^{N_x(\mathbf{T})}}
\end{equation}
For $h_1<2/d$ and large $t_1$, with positive probability, $H(z^*)-H(t_1)\in (e^{-h_1t_1}/2,e^{-h_1t_1})$. Let $h_2<h_1$, we define a stopping time $T_H$ by 
  $$T_H:=\inf\{s\geq t_1: H(z^*)-H(s)\geq e^{-h_2 s}\}$$
By Doob's inequality in $L^2$ and (\ref{gHbdCh}), for $n \geq t_1$, for the stopped martingale $M_{H\circ \pi}^{T_H}$ where $M_{H\circ \pi}^{T_H}(t):=M_{H\circ \pi}(T_H\wedge t)$, we have 
  \begin{equation}
    \label{bdMstopped}
    \begin{aligned}
        \mathbb{P}\left(\sup _{s\geq n}\left|M_{H\circ \pi}^{T_H}(s)-M_{H\circ \pi}^{T_H}(n)\right| \geq  e^{-\frac{nh_3}{2}}\right) &\leq 4e^{nh_3} \EE \int_n^{T_H}g_{H\circ \pi}(X_u,T(u))du \\
        &\leq C e^{-n(h_2+1-\frac{1}{d}-h_3)}
    \end{aligned}
  \end{equation}
  where $h_3<h_2+1-\frac{1}{d}$ is a constant. Note that for any $n\leq t<n+1$,
$$\left|M_{H \circ \pi}^{T^m}(\infty)-M_{H \circ \pi}^{T^m}(t)\right| \leq\left|M_{H \circ \pi}^{T^m}(\infty)-M_{H \circ \pi}^{T^m}(n)\right|+\left|M_{H \circ \pi}^{T^m}(t)-M_{H \circ \pi}^{T^m}(n)\right|$$
Thus, by (\ref{bdMstopped}), conditional on $\FF_{t_1}$, with probability $1-o(t_1)$,
$$
\left|M_{H\circ \pi}^{T_H}(\infty)-M_{H\circ \pi}^{T_H}(t)\right|\leq C e^{-\frac{h_3t}{2}}, \quad \forall t\geq t_1
$$
We denote this event by $E_1$. Thus, for $t_1\leq t< T_H$, by Lemma \ref{QKd}, on $E_1$,
\begin{equation}
  \label{remainbdh3}
  |-Q\nabla H\circ \pi(\mathbf{T}(t))_{X_{t}}+M_{H\circ \pi}(\infty)-M_{H\circ \pi}(t)-\int_{t}^{\infty} \frac{\partial}{\partial T_{X_{u}}} Q\nabla H\circ \pi(\mathbf{T}(u))_{X_{u}} d u| \leq C e^{-\frac{h_3t}{2}}
\end{equation}
By possibly choosing a larger $t_1$, we may assume that $\pi(\mathbf{T}(t_1)) \in B(z^*,\delta/2)$ where $\delta$ is as in  Proposition \ref{Hjconst} with $C_1\in (h_1,2/d)$ and $C_2\in (2/d,h_3/2)$ which is possible since $2/d<(d-1)/d$ for $d>3$. We write $R(u):=H(z^*)-H(u)$, then by (\ref{stoappequHinft}), on $E_1$, for $t\in [t_1,T_H)$,
 $$
  R(t)\geq C_1\int_t^{\infty}R(u)du - C e^{-\frac{h_3t}{2}}
    $$
Write $y(t):=e^{C_1 t} \int_t^{\infty} R(u) d u$. We have $y^{\prime}(t) \leq C e^{-(\frac{ h_3}{2}-C_1)t}$. Thus, $y(t)$ is upper bounded on $[t_1,T_H)$, and $
\int_t^{\infty}R(u)du \leq C e^{-C_1t}$. Then, by (\ref{stoappequHinft}) and (\ref{remainbdh3}), for $t\in [t_1,T_H)$,
\begin{equation}
  \label{RtbdC2C}
  H\left(z^*\right)-H(t)=R(t) \leq C_2 \int_t^{\infty} R(u) d u+C e^{-\frac{h_3t}{2}} \leq C e^{-C_1 t}
\end{equation}
By possibly choosing a larger $t_1$, we see that $T_H=\infty$ on $E_1$. Now, write $y_1(t):=e^{C_2 t} \int_t^{\infty} R(u) d u$. On $E_1$, for $t\geq t_1$
\begin{equation}
  \label{ytlowbd}
  y_1^{\prime}(t) \geq -C e^{-(\frac{h_3}{2}-C_2)t}; \quad  y(t) \geq y(t_1) - \frac{C}{\frac{h_3}{2}-C_2}  e^{-(\frac{h_3}{2}-C_2)t_1}
  \end{equation}
 Since $R(t_1)>e^{-h_1t_1}/2$, by  (\ref{RtbdC2C}) and possibly choosing a larger $t_1$, the right-hand side of the second inequality in (\ref{ytlowbd}) is positive. We denote it by $C_1(t_1)$. In particular, on $E_1$,
  \begin{equation}
   \label{intRlowbound}
 \int_t^{\infty}R(u)du \geq C_1(t_1) e^{-C_2t} , \quad t\geq t_1
 \end{equation}
On $E_1$, by (\ref{remainbdh3}), we have, for some positive constant $C_2(t_1)$,
\begin{equation}
  \label{RtlwC2}
  R(t)\geq C_1 \int_t^{\infty}R(u)du - C  e^{-\frac{h_3t}{2}}\geq C_2(t_1) e^{-C_2t}, \quad t\geq t_1
\end{equation}
Since $C_2$ could be arbitrarily close to $2/d$, the first inequality in (\ref{opticVRRW}) follows from Proposition \ref{Hjconst}.

Again, we denote by $Y_{x}(t)$ the number of visits of $(X_t)_{t\geq 0}$ to $x\in V$ up to time $t$. By (\ref{lilMu}), a.s. 
$$
|\frac{Y_i(t)}{\sum_{x\in V} Y_x(t)}-\frac{W_ie^{N_i(\mathbf{T}(t))}}{\sum_{x\in V} W_x e^{N_x(\mathbf{T}(t))}}|\leq C(\omega)e^{-\frac{2}{5}(1-\frac{1}{d})t}
$$
By (\ref{RtlwC2}), on $E_1$, 
$$
\max_{i\in V} \{|\frac{W_ie^{N_i(\mathbf{T}(t))}}{\sum_{x\in V} W_x e^{N_x(\mathbf{T}(t))}}-\frac{1}{d}|\} \geq \sqrt{\frac{C_2(t_1)}{3}}e^{-\frac{C_2t}{2}}
$$
Since $d>3$, we may choose $C_2$ to be close to $2/d$ such that $\frac{C_2}{2}<\min\{\frac{2}{5}(1-\frac{1}{d}), \frac{1}{d}+\frac{(d-1)\varepsilon}{d}\}$. Then, noting that 
$$
|\frac{Y_i(t)}{\sum_{x\in V} Y_x(t)}-\frac{1}{d}|\geq |\frac{W_ie^{N_i(\mathbf{T}(t))}}{\sum_{x\in V} W_x e^{N_x(\mathbf{T}(t))}}-\frac{1}{d}| -|\frac{Y_i(t)}{\sum_{x\in V} Y_x(t)}-\frac{W_ie^{N_i(\mathbf{T}(t))}}{\sum_{x\in V} W_x e^{N_x(\mathbf{T}(t))}}|
$$
we have, by Theorem \ref{cvrrwcomlike} and (\ref{llnYW}), on $E_1$, for $t\geq t_1$,
$$ \lim_{t\to\infty} (\sum_{j\in V}Y_j(t))^{\frac{1}{d-1}+\varepsilon} \max_{i\in V}\{ |\frac{Y_i(t)}{\sum_{x\in V} Y_x(t)}-\frac{1}{d}|\}=\infty$$
Then, the second inequality in (\ref{opticVRRW}) now follows from Proposition \ref{vrrwmixcvrrw}.

\end{proof}

\bibliographystyle{plain}
\bibliography{math_ref}

\begin{thebibliography}{10}

\bibitem{aldous-fill-2014}
David Aldous and James~Allen Fill.
\newblock Reversible markov chains and random walks on graphs, 2002.
\newblock Unfinished monograph, recompiled 2014, available at
  \url{http://www.stat.berkeley.edu/$\sim$aldous/RWG/book.html}.

\bibitem{MR1428513}
Michel Bena\"{\i}m.
\newblock Vertex-reinforced random walks and a conjecture of {P}emantle.
\newblock {\em Ann. Probab.}, 25(1):361--392, 1997.

\bibitem{MR2932667}
Michel Bena\"{\i}m and Pierre Tarr\`es.
\newblock Dynamics of vertex-reinforced random walks.
\newblock {\em Ann. Probab.}, 39(6):2178--2223, 2011.

\bibitem{MR3443368}
Samuel~N. Cohen and Robert~J. Elliott.
\newblock {\em Stochastic calculus and applications}.
\newblock Probability and its Applications. Springer, Cham, second edition,
  2015.

\bibitem{MR1030727}
Burgess Davis.
\newblock Reinforced random walk.
\newblock {\em Probab. Theory Related Fields}, 84(2):203--229, 1990.

\bibitem{MR0838085}
Stewart~N. Ethier and Thomas~G. Kurtz.
\newblock {\em Markov processes}.
\newblock Wiley Series in Probability and Mathematical Statistics: Probability
  and Mathematical Statistics. John Wiley \& Sons, Inc., New York, 1986.
\newblock Characterization and convergence.

\bibitem{MR0520004}
Dominique L{\'e}pingle.
\newblock Sur le comportement asymptotique des martingales locales.
\newblock In {\em S\'{e}minaire de {P}robabilit\'{e}s, {XII} ({U}niv.
  {S}trasbourg, {S}trasbourg, 1976/1977)}, volume 649 of {\em Lecture Notes in
  Math.}, pages 148--161. Springer, Berlin, 1978.

\bibitem{MR2759737}
Vlada Limic and Stanislav Volkov.
\newblock V{RRW} on complete-like graphs: almost sure behavior.
\newblock {\em Ann. Appl. Probab.}, 20(6):2346--2388, 2010.

\bibitem{MR0714271}
V.~Losert and E.~Akin.
\newblock Dynamics of games and genes: discrete versus continuous time.
\newblock {\em J. Math. Biol.}, 17(2):241--251, 1983.

\bibitem{MR0873887}
M.~M\'{e}tivier and P.~Priouret.
\newblock Th\'{e}or\`emes de convergence presque sure pour une classe
  d'algorithmes stochastiques \`a pas d\'{e}croissant.
\newblock {\em Probab. Theory Related Fields}, 74(3):403--428, 1987.

\bibitem{pemantle1988thesis}
R.~Pemantle.
\newblock {\em Random Processes with Reinforcement}.
\newblock PhD thesis, Massachusettes Institute of Technology, 1988.

\bibitem{MR1083357}
Daniel Revuz and Marc Yor.
\newblock {\em Continuous martingales and Brownian motion}, volume 293 of {\em
  Grundlehren der mathematischen Wissenschaften [Fundamental Principles of
  Mathematical Sciences]}.
\newblock Springer-Verlag, Berlin, 1991.

\bibitem{MR3420510}
Christophe Sabot and Pierre Tarr\`es.
\newblock Edge-reinforced random walk, vertex-reinforced jump process and the
  supersymmetric hyperbolic sigma model.
\newblock {\em J. Eur. Math. Soc. (JEMS)}, 17(9):2353--2378, 2015.

\bibitem{tarres2011localization}
Pierre Tarrès.
\newblock Localization of reinforced random walks.
\newblock {\em arXiv preprint arXiv:1103.5536}, 2011.

\bibitem{volkov2001vertex}
Stanislav Volkov.
\newblock Vertex-reinforced random walk on arbitrary graphs.
\newblock {\em The Annals of Probability}, 29(1):66--91, 2001.

\end{thebibliography}

\end{document}